\documentclass[12pt]{amsart}

\usepackage{amssymb, a4wide, mathdots,url}
\usepackage{verbatim}
\usepackage[all]{xy}

\usepackage{braket}

\usepackage[usenames]{xcolor}

\usepackage[colorlinks,pagebackref=true]{hyperref}
\hypersetup{colorlinks,linkcolor={blue},citecolor={blue},urlcolor={red}}

\newcommand{\supp}{{\rm Supp}}
\newcommand{\Alg}{\Pi}

\newcommand{\ip}{{\bf i}}
\newcommand{\triv}{\mathbf{1}}
\newcommand{\ms}{{\rm MS}}
\newcommand{\fin}{{\rm fin}}
\newcommand{\zz}{\mathbb Z}
\newcommand{\pp}{\mathcal P}

\newcommand{\rec}{{\rm rec}}
\newcommand{\aaa}{\mathcal A}
\newcommand{\weil}{\mathcal G}

\newcommand{\cO}{\mathcal O}

\newcommand{\Z}{\mathbb{Z}}

\newcommand{\wnu}{\nu}

\newcommand{\mult}{\mathfrak{m}}

\newcommand{\Hom}{\operatorname{Hom}}
\renewcommand{\subset}{\subseteq}

\newcommand{\GL}{\operatorname{GL}}
\newcommand{\Sp}{\operatorname{Sp}}

\newtheorem{theorem}{Theorem}[section]
\newtheorem{lemma}[theorem]{Lemma}

\newtheorem{prop}[theorem]{Proposition}
\newtheorem{remark}[theorem]{Remark}

\newtheorem{definition}[theorem]{Definition}
\newtheorem{defn}[theorem]{Definition}

\newtheorem{corr}[theorem]{Corollary}

\newtheorem{example}[theorem]{Example}
\newtheorem{hypothesis}[theorem]{Hypothesis}

\title {Models of Representations and Langlands Functoriality}

\author{Arnab Mitra}
\address{Tata Institute of Fundamental Research, Mumbai, India}
\email{00.arnab.mitra@gmail.com}
\author{Eitan Sayag}
\address{Department of Mathematics, Ben-Gurion University of the Negev ,  P.O.B. 653,
B'eer Sheva 84105, Israel}
\email{eitan.sayag@gmail.com}

\thanks{Arnab Mitra, partially supported by postdoctoral fellowships funded by the Department of Mathematics, Technion.}

\begin{document}

\setcounter{tocdepth}{1}
\keywords{Local Langlands correspondence, Klyachko models, degenerate Whittaker models}
\subjclass[2010]{22E50 (Primary); 11F70 (Secondary)}

\begin{abstract}
In this article we explore the interplay between two generalizations of the Whittaker model, namely the Klyachko models and the degenerate Whittaker models, and two functorial constructions, namely base change and automorphic induction, for the class of unitarizable and ladder representations of the general linear groups. 
\end{abstract}

\maketitle
\tableofcontents

\section{Introduction}
Let $F$ be a non-archimedean local field. Let $G$ be a quasi-split reductive group with a Borel subgroup $B$ defined over $F$. Let $U$ denote the unipotent radical of $B$ and $\psi$ a fixed non-degenerate character of it. A smooth irreducible representation $(\pi,V)$ of $G$ is said to have a Whittaker model, or to be generic, if there exists a non-trivial linear functional $\ell$ on $V$ such that $\ell(\pi(u)v)=\psi(u)\ell(v)$ for all $u\in U$ and $v\in V$. The importance of Whittaker model in the theory of automorphic forms cannot be overstated. However, not every irreducible unitarizable representation of $G$ admits a Whittaker model. To overcome this one needs to consider other models which contain non-generic irreducible representations. 

In the current article we focus on representations of the general linear groups and two families of models containing the Whittaker model-the degenerate Whittaker models and the Klyachko models. A degenerate Whittaker model is defined by allowing the character $\psi$ of $U$ in the definition of the Whittaker model to be arbitrary. They were introduced and studied in \cite[\S 8.3]{Z} by A. V. Zelevinsky. In particular, he showed that given any irreducible representation of $\GL_{n}(F)$, the representation admits a degenerate Whittaker model and does so with multiplicity one. 

The second family of models were introduced by M. J. Heumos and S. Rallis  in \cite{HR} (see \S \ref{defkt} for the definitions), inspired by the work of A. A. Klyachko for the groups $\GL_{n}(\mathbb F_{q})$ (where $\mathbb F_{q}$ is the finite field with $q$ elements). Although they provided examples of irreducible representations that do not admit any Klyachko model, there are many non-generic irreducible representations that do. For instance every unitarizable representation of $\GL_{n}(F)$ admits a Klyachko model (see \cite[Theorem 3.7]{OS1}). It was shown in \cite[Theorem 1]{OS2} that any irreducible representation which admits a Klyachko model, admits a unique Klyachko model and with multiplicity one. Thus to any irreducible representation $\pi$ of $\GL_{n}(F)$ which admits a Klyachko model, we can assign a unique integer between $0$ and $n$ indicating the precise Klyachko model it admits. We denote this integer by $r(\pi)$ and call it the Klyachko type of $\pi$. 

The local Langlands correspondence gives a bijection between the set of equivalence classes of  irreducible admissible representations of ${\rm GL}_{n}(F)$ and the set of equivalence classes of  $n$-dimensional Weil-Deligne representations of the Weil group of $F$. Let $E$ be a finite Galois extension of $F$ of degree $d$. Denote by $\mathcal A_{F}(n)$ and $\mathcal A_{E}(n)$ the set of all equivalence classes of irreducible representations of ${\rm GL}_{n}(F)$ and ${\rm GL}_{n}(E)$ respectively. The Weil-Deligne group of $E$, denoted by $W_{E}'$,  naturally sits as a subgroup of the Weil-Deligne group of $F$, $W_{F}'$. Via the correspondence, one can assign an irreducible representation of the general linear group over one field to a given irreducible representation of the general linear group over the other, by employing functorial constructions on the corresponding Weil-Deligne representations. In this paper, we deal with two such constructions. The base change map ${\rm bc}_{E/F}: \mathcal A_{F}(n)\to \mathcal A_{E}(n)$ is obtained by restricting the corresponding Weil-Deligne representation of the Weil group of $F$. On the other hand, the automorphic induction map ${\rm ai}_{E/F}: \mathcal A_{E}(n)\to \mathcal A_{F}(dn)$ is obtained by inducing up the corresponding $n$-dimensional Weil-Deligne representation of the Weil group of $E$. In \cite{AC} J. Arthur and L. Clozel investigated the first map while G. Henniart and R. Herb investigated the second in \cite{HH}.

In this paper, we investigate the effect of base change and automorphic induction on the two generalizations of the Whittaker model mentioned above-the degenerate Whittaker models and the Klyachko models, for certain classes of irreducible admissible representations. 

\subsection{Main results}
We now describe our main findings in more detail. Henceforth in this section we fix a Galois extension $E/F$ of non-archimedean local fields of characteristic different than two such that $d=[E:F]$ is prime.

Before we state our main results we need to introduce some more terminology. Call an irreducible representation {\it rigid} if it is supported on a single cuspidal line (see Definition \ref{def: cusp supp_1}). For $\pi\in \mathcal A_{F}(n)$, two partitions of the integer $n$ were defined in \cite{OS} and \cite{Z} respectively which we denote by $\mathcal V(\pi)$ and ${\bf d}(\pi)$ (see Definition \ref{defn: sl2} and \S \ref{sss:def_depth} for the respective definitions). The partition $\mathcal V(\pi)$ is called the $SL(2)$-type of the representation $\pi$. We begin with the following result which investigates its effect on the ${\rm bc}_{E/F}$ and ${\rm ai}_{E/F}$ maps.
\begin{theorem}[See Theorem \ref{mainth}]\label{mainth_3}
\begin{enumerate}
\item Let $\pi\in \mathcal A_{F}(n)$ be a rigid representation. Then $\mathcal V(\pi)=\mathcal V({\rm bc}_{E/F}(\pi))$.
\item Let $\Pi\in \mathcal A_{E}(n)$ be a rigid representation. Then $d\mathcal V(\Pi)=\mathcal V({\rm ai}_{E/F}(\Pi))$.
\end{enumerate}
\end{theorem}

The degenerate Whittaker model that an irreducible representation $\pi$ admits as per the prescription in \cite{Z} is with respect to the composition ${\bf d}(\pi)$. Using Theorem \ref{mainth_3} we get the following relationship between degenerate Whittaker models and the two maps.
\begin{theorem}[See Theorem \ref{deg_whit_func}]\label{main_deg_whit_func_1}
\begin{enumerate}
\item Let $\pi\in \mathcal A_{F}(n)$ be a rigid representation. Then ${\rm bc}_{E/F}(\pi)$ has a degenerate Whittaker model given by the depth sequence ${\bf d}(\pi)$.
\item Let $\Pi\in \mathcal A_{E}(n)$ be a rigid representations. Then ${\rm ai}_{E/F}(\pi)$ has a degenerate Whittaker model given by the depth sequence $\underbrace{\operatorname{{\bf d}(\pi)+_{c}\cdots+_{c}{\bf d}(\pi)}}\limits_{d-\text{times}}$.
\end{enumerate}
\end{theorem}
(Here the composition $\underbrace{\operatorname{{\bf d}(\pi)+_{c}\cdots+_{c}{\bf d}(\pi)}}\limits_{d-\text{times}}$ denotes the composition of $nd$ obtained by coordinate wise adding $d$ copies of the composition ${\bf d}(\pi)$.)

{\it Ladder} representations (see \S \ref{sss: ladder} for the definition) are a class of irreducible representations of general linear groups over non-archimedean local fields. The building blocks of the unitarizable dual of the general linear groups, the so called Speh representations (see \S \ref{sss:tadic}), constitute a subset of the ladders. Recall from above that to any representation $\pi\in \mathcal A_{F}(n)$ that admits a Klyachko model, we assign a unique integer $r(\pi)$ ($0\leq r(\pi)\leq n$) indicating the precise Klyachko model $\pi$ admits. Next we have the following relationship between the Klyachko models and the two maps.
\begin{theorem}[See Theorem \ref{Klyachkoai}]\label{Klyachkoai_3}
\begin{enumerate}
\item Let $\pi\in\mathcal A_{F}(n)$ be a ladder representation. Then $\pi$ admits a Klyachko model if and only if ${\rm bc}_{E/F}(\pi)$ admits one. Moreover 
\[
r({\rm bc}_{E/F}(\pi))=r(\pi).
\]  
\item Let $\Pi\in\mathcal A_{E}(n)$ be a ladder representation. Then $\Pi$ admits a Klyachko model if and only if ${\rm ai}_{E/F}(\Pi)$ admits one. Moreover $ r({\rm ai}_{E/F}(\Pi))=dr(\Pi)$.  
\end{enumerate}
\end{theorem}

While Theorem \ref{Klyachkoai_3} shows that the two maps ``preserve" the Klyachko type of a representation in the ladder class if it exists, we study yet another indicator of compatibility. Let us consider the case of the base change map and let $\Pi\in \mathcal A_{E}(n)$ be a ladder representation in the image of the map, which admits a Klyachko model. Any rigid representation in the fiber of $\Pi$ also admits the Klyachko model of the same type as $\Pi$ (by Lemma \ref{strucpres}(1) and Theorem \ref{thm: kly lad}) although there are many non rigid representations in the fiber which admit a different Klyachko model or none at all. Thus, given a representation satisfying the conditions that we imposed on $\Pi$ above, one might ask what proportion of the representations in its fiber admit the corresponding Klyachko model. Our next result analyses this question. For the sake of simplicity, we only state a special case of our result here, and just for the base change map. We refer the reader to Theorem \ref{prop_fiber:maps_all} for the result in its full generality and for its automorphic induction analogue.
    
\begin{theorem}[See Lemma \ref{fiber:maps}, Remark \ref{rem_fiber:maps}, Lemma \ref{lem_fiber:maps_gen} and Lemma \ref{lem_fiber:maps_symp}]\label{prop_fiber:maps_all_1} 
Suppose that $\Pi=L(\mathfrak m)$ (see \S \ref{sss: Langlands_class} for the notation) is a rigid  representation of $\GL_{n}(E)$ such that it is in the image of the base change map. Denote by $s$ the size of the multi-set $\mathfrak m$. 
Then we have the following:
\begin{enumerate}
\item The set ${\rm bc}_{E/F}^{-1}(\Pi)$ has cardinality $d^{s}$.
\item The representation $\Pi$ is generic if and only if every element in its fiber is so. 
\item Further suppose that $\Pi$ is a ladder representation. If it has a symplectic model, then the number of representations in its fiber that admit a symplectic model is $d^{\frac{s}{2}}$.
\end{enumerate}
\end{theorem}

\subsection{Context and related works}\label{ss:rel_works}
The class of ladder representations play an important role in this article. This class of irreducible representations was introduced and studied by E. Lapid and A. M\'inguez in \cite{LM}. They proved several results on the structural properties of the standard modules of ladder representations, which makes them easier to deal with than general irreducible representations. For instance a very useful tool at ones disposal when working with ladders representations (but one that is not available for general irreducible representations) is an explicit description of their Jacquet modules. This was obtained in \cite{KL12} using the results of \cite{LM}. At the same time the ladder class consists of many interesting examples of representations, for instance the Speh representations, as mentioned earlier. This makes them an ideal class of representations to test the plausibility of conjectures for the entire admissible dual. The classification of ladder representations with respect to the Klyachko models was recently obtained in \cite{MOS}.

The result on compatibility of base change and Klyachko models for the class of unitarizable representations was obtained in \cite{OS}. There it was shown that the $SL(2)$-type of a unitarizable representation is preserved under the operation of base change. This statement was then used to show that Klyachko types of unitarizable representations are invariant under base change. 

We obtain here independent proofs of the main results of \cite{OS}. Moreover, we also obtain the corresponding results for the automorphic induction map. We prove the statement about $SL(2)$-type for all rigid representations. However in this paper $SL(2)$-type does not play a role in the proof of the results underlying the connection between Klyachko models and base change. Instead we directly prove that base change preserves Klyachko type for ladder representations. The fact that any unitarizable representation can be obtained by inducing up Speh representations is then used to prove the statement for the unitarizable class. 

We remark that in a similar vein to this article the interplay of models of representations and base change was also studied in \cite{MV} where the model in question was a special case of the so called {\it linear models} for general linear groups. 

\subsection{Techniques of the proofs}
Recall that in \cite{Z} Zelevinsky classified the irreducible representations of the general linear groups in terms of the cuspidal representations. We begin by showing that  both the base change and automorphic induction maps are compatible with this classification and commute with the Zelevinsky involution, for the class of rigid representations (see Lemma \ref{strucpres} and Proposition \ref{mainpr} respectively). We show Theorem \ref{mainth_3} using Proposition \ref{mainpr}. In  Lemma \ref{sl_depth} we observe then that, for $\pi\in \mathcal A_{F}$, we have $\mathcal V(\pi)={\bf d}(\pi)^{t}$. This lemma is the non-archimedean analogue of \cite[Theorem 2.4.2]{GOSS}. Theorem \ref{main_deg_whit_func_1} is an easy consequence of Theorem \ref{mainth_3} and  Lemma \ref{sl_depth}. 

The classification results obtained in \cite{MOS} for ladders (Theorem \ref{thm: kly lad} in this article) play a critical role in the proof of Theorem \ref{Klyachkoai_3} which is the central result of this article. It follows directly from Theorem \ref{thm: kly lad} that whether or not a ladder representation admits a Klyachko model is independent of the cuspidal line it is supported on and depends only on the `shape' of the representation. Lemma \ref{strucpres} says that both these maps take a ladder representation to a product of ladders each having the same `shape' and supported on pairwise disjoint cuspidal lines. The proof of Theorem \ref{Klyachkoai_3} is based on this fact. Our proofs for the analogous results for the unitarizable class is based on the results for ladders, as described in \S \ref{ss:rel_works}.  

Theorem \ref{thm: kly lad} is also the key ingredient of the proof of Theorem \ref{prop_fiber:maps_all_1} (and that of the more general Theorem \ref{prop_fiber:maps_all}). We obtain a general description of the fiber of a rigid representation under the two maps in Lemma \ref{fiber:maps} which is then used along with Theorem \ref{thm: kly lad} to demonstrate Theorem \ref{prop_fiber:maps_all_1}.

\subsection{Organization of the paper}
The rest of the paper is organized as follows. In \S \ref{not_pre} we set up some general notation while in \S \ref{s:pre_rep_gl} we review some preliminaries on the irreducible representations of $\GL_{n}(F)$ and the Weil-Deligne representations. In \S \ref{defbcai} we formally define the base change and automorphic induction maps using the reciprocity map and prove some basic results used in the sequel, including their compatibility with the Zelevinsky classification. In \S \ref{s_lad} we recall the preliminaries of ladder representations. In \S \ref{relsl} we demonstrate our results on $SL(2)$-type, namely Theorem \ref{mainth_3} (see Theorem \ref{mainth}). We then use it to study the relationship of the degenerate Whittaker models with the two maps, and in particular, prove Theorem \ref{main_deg_whit_func_1} (see Theorem \ref{deg_whit_func}). In \S \ref{relkts} we prove Theorem \ref{Klyachkoai_3} (see Theorem \ref{Klyachkoai}). In \S \ref{s:fiber} we analyze the fiber of the two maps with respect to the Klyachko models proving a general version of Theorem \ref{prop_fiber:maps_all_1} (see Theorem \ref{prop_fiber:maps_all}) and its automorphic induction analogue.  

\section{Notation}\label{not_pre}
We set some primary notation in this section. More particular notation is defined in the section where it first occurs. 

\subsubsection{}\label{ss:not_pre} 
Let $F$ be a non-archimedean local field of characteristic different than two, $\mathcal O_{F}$ be the ring of integers of $F$, $\mathfrak p_{F}$ be the unique prime ideal of $\mathcal O_{F}$, and $\varpi_{F}$ be a fixed choice of a uniformizer of the prime ideal. Let $q_{F}$ denote the cardinality of its residue field. 

Let $|.|_{F}:F^{\times}\to \mathbb C^{\times}$ denote the standard absolute value normalized so that $|\varpi_{F}|_{F}=q_{F}^{-1}$. The character of $\GL_{n}(F)$ given by $g \mapsto |\det(g)|_{F}$ is denoted by $\nu_{F}$. Let $W_{F}$ and $W_{F}'$ denote the Weil group and the Weil-Deligne group of $F$ respectively. The reciprocity map $T_{F}: W_{F}\to F^{\times}$, which is given by the local class field theory, is normalized such that geometric Frobenius elements are mapped to uniformizers. The map $T_{F}$ defines an isomorphism from the abelianization $W_{F}^{ab}$ of $W_{F}$ to $F^{\times}$. The composition $|.|_{F}\circ T_{F}$ gives the associated absolute value on $W_{F}$ which we denote by $||\cdot||$. We will also denote it sometimes by $\wnu_{F}'(\cdot)$ when we wish to highlight its analogy with the character $\nu_{F}$.

\subsubsection{Classes of representations}
The category of smooth complex valued representations of a topological group $G$ of finite length will be denoted by $\Pi(G)$. Denote by $\mathcal A_{F}(n)$ the class of all irreducible representations in $\Pi({\rm GL}_{n}(F))$ and by $\mathcal A_{F}$, the union $\sqcup _{n\geq 1}\mathcal A_{F}(n)$. Let $\mathcal A_{F}^{\circ}(n)$ and $\mathcal A_{F}^{u}(n)$ be the subset of  $\mathcal A_{F}(n)$ consisting of the cuspidal representations and the unitarizable representations respectively. Further let $\mathcal A_{F}^{\circ}$ and $\mathcal A_{F}^{u}$ denote the corresponding unions. 

\subsubsection{The Bernstein-Zelevinsky product}
Set $G={\rm GL}_{n}(F)$. Let $M$ be the $F$-points of a standard Levi subgroup of $\GL_{n}$. We will denote by $\ip_{G,M}$ the normalized parabolic induction functor from $\Alg(M)$ to $\Alg(G)$. Let $(n_{1},\dots,n_{k})$ be a composition of $n$ and let $\pi_i\in\Alg({\rm GL}_{n_{i}}(F))$, $i=1,\dots,k$.  Assume that $M\cong \Pi_{i=1}^{k}{\rm GL}_{n_{i}}(F)$. Let $\pi:=\pi_1\otimes \cdots \otimes \pi_k$. Then $\pi\in \Alg(M)$. Set
\begin{equation*}
\pi_1\times\cdots \times \pi_k :=\ip_{G,M}(\pi).
\end{equation*}

\subsubsection{Distinguished representations}
This paper is concerned with distinguished representations in the following sense. 
\begin{definition}
Let $\pi$ be a smooth, complex-valued representation of $G$ and $H$ a closed subgroup of $G$. 
\begin{itemize}
\item We say that $\pi$ is $H$-distinguished if the space $\Hom_H(\pi,1)$ of $H$-invariant linear forms on $\pi$ is non-zero.
\item More generally, for a character $\chi$ of $H$ we say that $\pi$ is $(H,\chi)$-distinguished if the space $\Hom_H(\pi,\chi)$ is non-zero.
\end{itemize}
\end{definition}

\subsubsection{Generic representations}\label{sss:gen}
Denote by $U_{n}$ the $F$ points of the unipotent radical of the standard Borel subgroup of ${\rm GL}_{n}$. Let $\psi$ be a fixed non-trivial additive character of $F$. We further denote by $\psi_n$ the character of $U_n$ defined by
\begin{equation*}
\psi_{n}(u)=\psi(\sum_{i=1}^{n-1} u_{i,i+1}), \ u=(u_{i,j})\in U_n. 
\end{equation*}

\begin{defn}
Let $\pi$ be an irreducible representation of ${\rm GL}_{n}(F)$. We say that $\pi$ admits a Whittaker model, or is generic, if it is $(U_{n},\psi_{n})$-distinguished.
\end{defn}

\subsubsection{} We henceforth fix a cyclic extension $E/F$ of non-archimedean local fields of characteristic different than two such that $d=[E:F]$ is prime. Fix $\kappa=\kappa_{E/F}$ to be a character of $F^{\times}$ coming from the local class field theory with kernel equal to ${\rm N}_{E/F}(E^{\times})$ where ${\rm N}_{E/F}$ is the norm map from $E^{\times}$ to $F^{\times}$. Observe that $\nu_{E}(\cdot)=|{\rm N}_{E/F}({\rm det}(\cdot))|_{F}$ due to the normalization of the absolute values mentioned above.  

For $\pi\in \mathcal A_{E}(n)$ and an element $\gamma\in {\rm Gal}(E/F)$, denote the representation $\pi^{\gamma}\in \mathcal A_{E}(n)$ given by $\pi^{\gamma}(g)=\pi(\gamma(g)) \ \forall g\in {\rm GL}_{n}(E)$. 

\subsubsection{Multi-sets and partitions}
Denote by $\triv_\Omega$ the characteristic function of a set $\Omega$. Let $\ms_\fin(\Omega)$ be the set of finite multi-sets of elements in $\Omega$ i.e. the set of functions $f:\Omega \to\zz_{\geq 0}$ of finite support. When convenient we will also denote
$f$ by $\{\omega_1,\dots,\omega_1,\omega_2,\dots,\omega_2,\dots\}$ where $\omega\in \Omega$ is repeated $f(\omega)$ times. Let $\mathcal P=\ms_\fin(\mathbb Z_{>0})$ be the set of partitions of positive integers and let
$$ \mathcal P(n)=\{f \in \pp: \sum_{k=1}^\infty k\,f(k)=n\}$$
denote the subset of partitions of $n$. For $n,\,m \in \zz_{>0}$ let $(n)_m=m\,\triv_n=\{n,\dots,n\}$ be the partition of $nm$ with `$m$ parts of size $n$'. As indicated by the definition above, unless otherwise mentioned, we will not suppose a partition to be ordered. We will sometimes use the word `composition' in this article for an ordered partition.   

\section{Preliminaries on irreducible representations of $\GL_{n}$}\label{s:pre_rep_gl}
We now recall some basics of the representation theory of general linear groups over non-archimedean local fields. In this section $F$ will denote an arbitrary non-archimedean local field.

\subsection{Irreducible representations of $\GL_{n}(F)$}
\subsubsection{} For an irreducible cuspidal $\sigma\in \aaa_{F}^{\circ}$ define its {\it cuspidal line}
\begin{equation*}
\sigma^\Z=\{\nu_{F}^{m}\sigma\mid m\in \mathbb Z\}.
\end{equation*}

\subsubsection{}\label{sss:def_seg} We now recall the combinatorial notion of {\it segments} as introduced by Zelevinsky (in \cite{Z}), and briefly review the classification of $\aaa_{F}$.
\begin{defn}
Given an irreducible cuspidal representation $\sigma\in \aaa_{F}^{\circ}$ and $a,b \in \Z$ such that $a\leq b+1$, define the segment $[\nu^{a}\sigma,\nu^{b}\sigma]$ to be the set $\{ \nu^{a}\sigma, \nu^{a+1}\sigma,\dots,\nu^{b}\sigma \}$ if $a\leq b$ and the empty set if $a=b+1$. We say that the segment  $[\nu^{a}\sigma,\nu^{b}\sigma]$ is supported on $\sigma^\Z$.
\end{defn}

For a segment $\Delta=[\nu^{a}\sigma,\nu^{b}\sigma]=[a,b]_{(\sigma)}$, we denote by $b(\Delta)=\nu^{a}\sigma$ its beginning, by $e(\Delta)=\nu^{b}\sigma$ its end, and by $\ell(\Delta)=b-a+1$ its length respectively. 

The representation $\nu^{a}\sigma \times \cdots\times \nu^{b}\sigma$ has a unique irreducible subrepresentation and a unique irreducible quotient which we write as $Z(\Delta)$ and $L(\Delta)$ respectively. By convention, if the set $\Delta$ is empty, then both $Z(\Delta)$ and $L(\Delta)$ are defined to be the trivial representation of the trivial group.
\begin{defn}
Two segments $\Delta_{1}$ and $\Delta_{2}$ are said to be linked if $\Delta_{1}\nsubseteq \Delta_{2}$, $\Delta_{2}\nsubseteq \Delta_{1}$ and $\Delta_{1}\cup \Delta_{2}$ is also a segment. If $\Delta_{1}$ and $\Delta_{2}$ are linked and $b(\Delta_{1}\cup\Delta_{2})=b(\Delta_{1})$, then we say that $\Delta_{1}$ precedes $\Delta_{2}$ and write $\Delta_{1}\prec \Delta_{2}$.
\end{defn}
Let $\cO$ be the set of multi-sets of segments. Let $\mathfrak m=\{\Delta_1,\dots,\Delta_t\}\in\cO$. The integer $t$ will be known as the {\it size} of the multi-set $\mathfrak m$ and will be denoted by $|\mathfrak m|$. Any permutation $\varsigma$ of the set $\{1,\dots,t\}$ induces an arrangement of the segments of the multi-set $\mathfrak m$ which we call an {\it order} on $\mathfrak m$. An order on $\mathfrak m$ is of \emph{standard} form if $\Delta_{\varsigma(i)}\not \prec\Delta_{\varsigma(j)}$ for all $i<j$. Clearly every $\mathfrak m\in \cO$ admits an order that is of standard form. 

\subsubsection{The Zelevinsky classification} Let $\mathfrak m=\{\Delta_1,\dots,\Delta_t\}\in\cO$ be ordered in standard form. The representation
\begin{equation*}
Z(\Delta_1) \times\cdots \times Z(\Delta_t)
\end{equation*}
is independent of the choice of order of standard form. It has a unique irreducible submodule that we denote by $Z(\mathfrak m)$.
 
The Zelevinsky classification says that the map $(\mathfrak m\mapsto Z(\mathfrak m)):\cO\rightarrow \aaa_{F}$ is a bijection (see \cite[Theorem 6.1]{Z}).

\subsubsection{The Langlands classification}\label{sss: Langlands_class}
Let $\mathfrak m=\{\Delta_1,\dots,\Delta_t\}\in\cO$ be ordered in standard form. The representation
\begin{equation*}
L(\Delta_1)\times\cdots\times L(\Delta_t)
\end{equation*}
is independent of the choice of order of standard form. It has a unique irreducible quotient that we denote by $L(\mathfrak m)$.

The Langlands classification says that the map $( \mathfrak m\mapsto L(\mathfrak m)) :\cO\rightarrow \aaa_{F}$ is a bijection (for instance see \cite[Theorem 1.2.5]{Kud}).

\subsubsection{The Zelevinsky involution}\label{mwalgo}
It follows from the two classifications above that for any $\mathfrak m\in \cO$ there exists a unique $\mathfrak m^t\in \cO$ such that $Z(\mathfrak m)=L(\mathfrak m^t)$. The function $\mathfrak m\mapsto \mathfrak m^t$ is an involution on $\cO$ known as the Zelevinsky involution. For $\pi=Z(\mathfrak m)\in \aaa_{F}$, let $\pi^t=L(\mathfrak m)$. Then $\pi\mapsto \pi^t$ is the corresponding involution on $\aaa_{F}$.

Given a multi-set $\mathfrak m$, an algorithm to compute $\mathfrak m^{t}$ is provided in \cite{MW}. 

\subsubsection{The cuspidal support}\label{def: cusp supp}
For every $\pi\in \aaa_{F}$ there exist $\sigma_1,\dots,\sigma_k\in \aaa_{F}^{\circ}$, unique up to rearrangement, so that $\pi$ is isomorphic to a subrepresentation of $\sigma_1\times \cdots \times \sigma_k$ (see \cite[Theorem 1.10]{Z}). Let $\supp(\pi)=\{\sigma_i:i=1,\dots,k\}$ be the support of $\pi$. For $\mathfrak m\in\cO$ let $\supp(\mathfrak m)=\{\sigma\in \mathcal A_{F}^{\circ}: \sigma\in\Delta\text{ for some }\Delta\in \mathfrak m\}$ be the support of $\mathfrak m$. \footnote{The support is often considered as a multi-set. In this article though only the underlying set plays a role and hence we will treat the support, of both a representation and a multi-set of segments, as a set.}

\begin{defn}\label{def: cusp supp_1}
A representation $\pi\in \aaa_{F}$ is said to be {\it rigid} if $\supp(\pi)\subseteq \sigma^\Z$ for some $\sigma\in \aaa_{F}^{\circ}$. A multi-set $\mathfrak m\in \cO$ is called rigid if $\supp(\mathfrak m)\subseteq\sigma^\Z$ for some $\sigma\in \aaa_{F}^{\circ}$. 
\end{defn}
In the sequel a rigid multi-set $\mathfrak m=\mathfrak m_{(\sigma)}=\{[\nu^{a_{1}}\sigma,\nu^{b_{1}}\sigma],\dots,[\nu^{a_{t}}\sigma,\nu^{b_{t}}\sigma]\}$ will sometimes be denoted by $\{[a_{1},b_{1}],\dots,[a_{t},b_{t}]\}_{(\sigma)}$. 

Let $\pi=\pi_{(\sigma)}\in \mathcal A_{F}$ be such that $\supp(\pi)\subset \sigma^{\mathbb Z}$. Write $\pi_{(\sigma)}=L(\mathfrak m_{(\sigma)})$ where $\mathfrak m_{(\sigma)}=\{[a_{1},b_{1}],\dots, [a_{t},b_{t}]\}_{(\sigma)}$. Let $F'$ be a non-archimedean local field (not necessarily different from $F$) and let $\sigma'\in \mathcal A_{F'}^{\circ}$. Then define $\pi_{(\sigma')}=L(\mathfrak m_{(\sigma')})=L(\{[a_{1},b_{1}],\dots, [a_{t},b_{t}]\}_{(\sigma')})$. (Note that $\{[a_{1},b_{1}],\dots, [a_{t},b_{t}]\}_{(\sigma')}$ is the multi-set $\{[\nu_{F'}^{a_{1}}\sigma',\nu_{F'}^{b_{1}}\sigma'],\dots, [\nu_{F'}^{a_{t}}\sigma',\nu_{F'}^{b_{t}}\sigma']\}$.) In particular for a segment $\Delta=[a,b]_{(\sigma)}$, we will denote by $\Delta_{(\sigma')}$ the segment $[\nu_{F'}^{a}\sigma',\nu_{F'}^{b}\sigma']$.

\subsection{The Weil-Deligne representations}
\subsubsection{Definition of Weil-Deligne representations}
\begin{defn}
An $n$-dimensional admissible Weil-Deligne representation of $W_{F}$ is a pair $((\rho,V),N)$ where $(\rho,V)$ is a semi-simple, smooth and complex valued representation of $W_{F}$ of dimension $n$ and the operator $N:V\to V$ is a nilpotent endomorphism such that 
\begin{equation}\label{defN}
\rho(w)\circ N\circ \rho(w)^{-1}=||w||N
\end{equation}
\end{defn}
where $||w||$ is the character of $W_{F}$ as defined in \S \ref{ss:not_pre}. A morphism of Weil-Deligne representations $((\rho_{1},V_{1}),N_{1})\to ((\rho_{2},V_{2}),N_{2})$ is a map of representations $T: (\rho_{1},V_{1})\to (\rho_{2},V_{2})$, such that $T\circ N_{1}=N_{2}\circ T$. For more details on Weil-Deligne representations we refer the reader to \cite{Wed}. 

Let $\mathcal G_{F}(n)$ denote the set of all isomorphism classes of $n$-dimensional admissible Weil-Deligne representations of $W_{F}$ and let $\mathcal G_{F}=\sqcup _{n\geq 0}\mathcal G_{F}(n)$.

\subsubsection{Classification of Weil-Deligne representations}\label{sss:classfn_wd_rep}
For a semi-simple, smooth and complex valued representation $(\rho,V)$ of $W_{F}$ put $\Delta'=[\nu^{'a}\rho,\nu^{'b}\rho]\ (a\leq b)$ to be the set $\{\nu^{'a}\rho,\nu^{'(a+1)}\rho,\dots,\nu^{'b}\rho\}$ as in \S \ref{sss:def_seg}. Let $\rho(\Delta')=\nu^{'a}\rho\oplus\nu^{'(a+1)}\rho \oplus\cdots \oplus\nu^{'b}\rho$. Let $V_{i}$ be the space on which $\nu^{'i}\rho$ acts ($a\leq i \leq b$). Clearly, all these spaces can be identified with the same space $V$. Define a map $N(\Delta'):\oplus_{i=a}^{b}V_{i}\to \oplus_{i=a}^{b}V_{i}$ in the following way. Let $N(\Delta'):V_{i}\to V_{i+1}$ be the obvious (identity) morphism (for $i=a,\dots, b-1$) and let $N(\Delta')|_{V_{b}}=0$. Now assign to each $\Delta'$ the Weil-Deligne representation $\tau(\Delta')=(\rho(\Delta'),N(\Delta'))$.

It follows from generalities that the $\tau(\Delta')$ are indecomposable objects in $\mathcal G_{F}$. They are mutually non-isomorphic and every indecomposable object in $\mathcal G_{F}$ is of this form. Thus every Weil-Deligne representation decomposes into a direct sum $\tau(\Delta_{1}')\oplus\cdots\oplus \tau(\Delta_{r}')$ (for some positive integer $r$), and moreover, this decomposition is unique up to a permutation. 

\subsubsection{The maps $\rec$ and $\rec^{\circ}$}
Let
\begin{equation*} 
\rec=\rec_{F}: \aaa_{F}\to \mathcal G_{F}
\end{equation*}
be the Langlands reciprocity map established in \cite{LRS} for the positive characteristic case and in \cite{HT} (and later also in \cite{H} and \cite{Sc}) for characteristic 0. Denote by $\rec^{\circ}$ the restriction of $\rec$ to $\aaa_{F}^{\circ}$. The map $\rec$ can be described in terms of the map $\rec^{\circ}$ as follows (see \cite[\S10]{Z} for details). If $\pi=L([\nu^{a_{1}}\sigma_{1},\nu^{b_{1}}\sigma_{1}],\dots,[\nu^{a_{t}}\sigma_{t},\nu^{b_{t}}\sigma_{t}])$, then
\begin{equation}\label{eq:rec}
\rec(\pi)=\oplus_{i=1}^{t}(\rho([\nu^{'a_{i}}\rec^{\circ}(\sigma_{i}),\nu^{'b_{i}}\rec^{\circ}(\sigma_{i})]), N([\nu^{'a_{i}}\rec^{\circ}(\sigma_{i}),\nu^{'b_{i}}\rec^{\circ}(\sigma_{i})])).
\end{equation}

\subsubsection{Partition associated to a Weil-Deligne representation}\label{sss:def_part}
Given an $n$-dimensional admissible Weil-Deligne representation $((\rho,V),N)$, one can associate a partition $f\in \pp(n)$ to it in the following manner. Since $N$ is a nilpotent endomorphism it can be written as a matrix with 1's on the sub-diagonal and 0's elsewhere in a unique way (up to the order of the Jordan blocks). Considering the size of the Jordan blocks of $N$ defines a partition of $n$ that we will denote by $f$. In particular, the partition corresponding to $N=0$ is the partition $n\mathbf{1}_{1}=\{1,\dots , 1\}$.

Denote by $P_{F,n}: \mathcal G_{F}(n) \to \pp(n)$ the map which takes $((\rho,V), N)$ to the partition $f$ as described above and let $P_{F}:\mathcal G_{F} \to \pp$ be the map such that ${P_{F}}|_{\mathcal G_{F}(n)}=P_{F,n}$.  The proof of the following result in an easy linear algebra exercise.
\begin{lemma}\label{lem_def_part}
Let $\pi=L([\nu^{a_{1}}\sigma_{1},\nu^{b_{1}}\sigma_{1}],\dots,[\nu^{a_{t}}\sigma_{t},\nu^{b_{t}}\sigma_{t}])$ where $\sigma_{i}\in \aaa_{F}^{\circ}(k_{i})$. Then we have
\begin{equation*}
P_{F}(\rec(\pi))=\sum_{i=1}^{t}(b_{i}-a_{i}+1)_{k_{i}}.
\end{equation*}
\qed
\end{lemma}

\subsubsection{} 
\begin{remark}
In \cite{OS} it was mistakenly remarked that the map $((\rho,V),N)\mapsto ([\rho],f)$ (where $[\rho]$ is the isomorphism class of the representation $\rho$) is an injection. Although this fact was used in the proofs of \cite[ Lemma 4.1 and Proposition 7.1]{OS}, it wasn't done so in a crucial manner, and they can be rectified simply by working with the description of the reciprocity map that we provide in eq.(\ref{eq:rec}) instead of the description given in \cite[\S 3]{OS}. In particular, the statements in that paper are correct. In any case, the results on ladder representations that we obtain in this article are used here to provide independent proofs of the main results of \cite{OS}.
\end{remark}

\section{Base change and automorphic induction}\label{defbcai}
The base change and the automorphic induction maps were studied in \cite{AC} and \cite{HH} respectively before the local Langlands correspondence for the general linear groups over non-archimedean local fields was established. Now that we have the correspondence at our disposal, these two maps can be defined in a much more simpler manner. We now recall these definitions. We also obtain some results analyzing the behavior of the class of rigid representations under these two maps. Some of these results (for instance Lemma \ref{strucpres}) can be found in the aforementioned references but we provide a proof here using the definitions of the two maps that we use in this article.     
\subsection{Definition of the two maps}
For now suppose that $E/F$ is an arbitrary finite Galois extension of non-archimedean local fields such that $[E:F]=d$.
\subsubsection{Base change}\label{defnbc}
Let $\pi\in \aaa_{F}(n)$ and ${\rm rec}_{F}(\pi)=((\rho,V),N)$. Denote by 
\[
{\rm res}_{E/F}:\weil_{n}(F)\to \weil_{n}(E)
\]
the map defined by ${\rm res}_{E/F}(\rho,N)=(\rho|_{W_{E}},N)$. This defines an irreducible representation of ${\rm GL}_{n}(E)$ via the local Langlands correspondence. The above process of obtaining an irreducible representation of ${\rm GL}_{n}(E)$ from an irreducible representation of ${\rm GL}_{n}(F)$ is known as {\it base change}. For $\pi\in\mathcal A_{F}(n)$, its base change will be denoted by ${\rm bc}_{E/F}(\pi)$ and is defined by 
\[
{\rm rec}_{E}({\rm bc}_{E/F}(\pi))={\rm res}_{E/F}({\rm rec}_{F}(\pi)).
\]
\subsubsection{Automorphic induction}\label{defnai}
Let $\pi\in \aaa_{E}(m)$ and ${\rm rec}_{E}(\pi)=((\rho,V),N)$. Define now the representation ${\rm ind}^{W_{F}}_{W_{E}}(\rho)$ of $W_{F}$ in the following way:
\[
{\rm ind}^{W_{F}}_{W_{E}}(\rho)=\{f:W_{F}\to V\ |\ f(hg)=\rho(h)f(g)\ \forall h\in W_{E}, \ g\in W_{F}\}.
\]
Since $\rho$ is semi-simple, the induced representation ${\rm ind}^{W_{F}}_{W_{E}}(\rho)$ is semi-simple as well. Further define $\tilde{N}$ such that $(\tilde{N}f)(g)=||g||N(f(g))$. It can be easily checked that $\tilde{N}$ is a nilpotent endomorphism of the induced space satisfying eq.(\ref{defN}). Thus define
\[
{\rm ind}_{E/F}((\rho, N))=({\rm ind}^{W_{F}}_{W_{E}}(\rho),\tilde{N}),
\]
an element in $\weil_{F}(md)$. This Weil-Deligne representation corresponds to an irreducible representation of ${\rm GL}_{md}(F)$, via the reciprocity map. This process of obtaining an irreducible representation of ${\rm GL}_{md}(F)$ from an irreducible representation of ${\rm GL}_{m}(E)$ is known as {\it automorphic induction}. For $\pi\in\mathcal A_{E}(m)$, its automorphic induction will be denoted by ${\rm ai}_{E/F}(\pi)$ and is defined by 
\[
{\rm rec}_{F}({\rm ai}_{E/F}(\pi))={\rm ind}_{E/F}({\rm rec}_{E}(\pi)).
\]
\subsubsection{} Our next lemma provides a simplified expression for the nilpotent operator $\tilde{N}$.
\begin{lemma}\label{patnai}
Let $\{ g_{1},\dots,g_{d}\}$ be a fixed set of representatives for the right coset space $H\setminus G$ where $G=W_{F}$ and $H=W_{E}$. Let $((\rho,V),N)\in \weil_{E}$ with $\tilde{N}:{\rm ind}^{G}_{H}(\rho)\to {\rm ind}^{G}_{H}(\rho)$ as defined in \S\ref{defnai}. Then we can choose bases of $V$ and ${\rm ind}^{G}_{H}(\rho)$ such that in the matrix form with respect to the two bases, $\tilde{N}={\rm diag}(||g_{1}||N,\cdots,||g_{d}||N)$. In particular, if the partition corresponding to the operator $N$ (via its Jordan canonical form) is $f$, then that corresponding to $\tilde{N}$ is $df$. 
\end{lemma}
\begin{proof}
Let ${\rm dim}\ V=l$ and $(v_{1},\dots,v_{l})$ be an ordered basis of $V$ such that the matrix of $N$ with respect to it is expressed in its Jordan form. Since $N$ is nilpotent, there exists $\{v_{l_{1}},\dots,v_{l_{k}}\}\subset \{v_{1},\dots,v_{l}\}$ such that 

\[
N(v_{j})=
\begin{cases} 
v_{j+1}& {\rm if}\ j\in \{l_{1},\dots,l_{k}\} \\
0& {\rm if}\ j\notin \{l_{1},\dots,l_{k}\}.
\end{cases}\\
\]

Now define a standard basis $\{ f_{i,j}\ | \ 1\leq i \leq d, 1\leq j\leq l \}$ of the space ${\rm ind}^{W_{F}}_{W_{E}}(\rho)$ in the following manner.
\[
f_{i,j}= 
\begin{cases}
\rho(h)v_{j} & {\rm if}\ g= hg_{i} \\
0& {\rm otherwise}.
\end{cases}
\]
Fix $i$. An easy calculation shows that
\[
\tilde{N}(f_{i,j})=
\begin{cases}
 ||g_{i}||f_{i,j+1} & {\rm if}\ j\in \{l_{1},\dots,l_{k}\} \\
0& {\rm otherwise}.
\end{cases}
\]
This gives the lemma.
\end{proof}

\subsection{Some preliminary results} 
We return to the case when $E/F$ is a cyclic extension of characteristic different than two such that $d=[E:F]$ is prime. Next we recall some basic properties of base change and automorphic induction.
\subsubsection{}The following result was obtained in \cite[Lemma 6.10]{AC} and \cite[Proposition 5.5]{HH} for the base change and the automorphic induction maps respectively. Assuming the correspondence one can obtain these results in an elementary fashion using arguments similar to the ones employed in the proof of \cite[Lemma 7.1]{OS}.   
\begin{lemma}\label{cuspfac}
\begin{enumerate}
\item Let $\sigma\in \mathcal A_{F}^{\circ}(k)$. Then $${\rm bc}_{E/F}(\sigma)=\sigma_{1}\times\cdots\times \sigma _{t}$$ where $t|k$ and $\sigma_{i}\in \mathcal A_{E}^{\circ}(\frac{k}{t})$ such that $\sigma_{i}\neq \nu_{E}^{\alpha}\sigma_{j}$, for any $\alpha\in \mathbb R$ and $i\neq j$. 
\item Analogously, let $\sigma\in \mathcal A_{E}^{\circ}(k)$. Then $${\rm ai}_{E/F}(\sigma)=\sigma_{1}\times\cdots\times \sigma _{t}$$ where $t|kd$ and $\sigma_{i}\in \mathcal A_{F}^{\circ}(\frac{kd}{t})$ such that $\sigma_{i}\neq \nu_{F}^{\alpha}\sigma_{j}$, for any $\alpha\in \mathbb R$ and $i\neq j$.
\item Moreover if $\sigma\in \mathcal A_{F}^{u}$ (resp., $\sigma\in \mathcal A_{E}^{u}$), then each $\sigma_{i}\ (i=1,\dots,t)$ appearing in ${\rm bc}_{E/F}(\sigma)$ (resp., ${\rm ai}_{E/F}(\sigma))$ is in $\mathcal A_{E}^{u}$ (resp., $\mathcal A_{F}^{u}$).  
\end{enumerate}
\qed
\end{lemma}

\subsubsection{Compatibility with parabolic induction}
\begin{lemma}\label{bcprod}
\begin{enumerate}
\item Let $\pi_{1},\dots,\pi_{r}\in \mathcal A_{F}$ such that both $\pi_{1}\times\cdots\times\pi_{r}$ and ${\rm bc}_{E/F}(\pi_{1})\times \cdots\times {\rm bc}_{E/F}(\pi_{r})$ are irreducible. Then 
$${\rm bc}_{E/F}(\pi_{1}\times\cdots\times\pi_{r})={\rm bc}_{E/F}(\pi_{1})\times \cdots\times {\rm bc}_{E/F}(\pi_{r}).$$
\item Let $\pi_{1},\dots,\pi_{r}\in \mathcal A_{E}$ such that both $\pi_{1}\times\cdots\times\pi_{r}$ and ${\rm ai}_{E/F}(\pi_{1})\times \cdots\times {\rm ai}_{E/F}(\pi_{r})$ are irreducible. Then 
\[
{\rm ai}_{E/F}(\pi_{1}\times\cdots\times\pi_{r})={\rm ai}_{E/F}(\pi_{1})\times \cdots\times {\rm ai}_{E/F}(\pi_{r}).
\]
\end{enumerate}
\end{lemma}
\begin{proof}
We first prove (1). The general case reduces to the case when $r=2$ by induction. Thus let $r=2$. Let ${\rm rec}_{F}(\pi_{1})=(\rho_{1},N_{1})$ and ${\rm rec}_{F}(\pi_{2})=(\rho_{2},N_{2})$. Then we have,
\begin{equation*}
\begin{split}
{\rm rec}_{E}({\rm bc}_{E/F}(\pi_{1})\times {\rm bc}_{E/F}(\pi_{2})) & ={\rm rec}_{E}({\rm bc}_{E/F}(\pi_{1}))\oplus{\rm rec}_{F}({\rm bc}_{E/F}(\pi_{2}))\\
&= (\rho_{1}|_{W_{E}},N_{1})\oplus (\rho_{2}|_{W_{E}},N_{2})\\
&= ((\rho_{1}\oplus\rho_{2})|_{W_{E}},N_{1}\oplus N_{2}).
\end{split}
\end{equation*}
This is equal to ${\rm rec}_{E}({\rm bc}_{E/F}(\pi_{1}\times\pi_{2}))$ which demonstrates the statement in the base change case.

Next we consider the statement for the automorphic induction case. As above it is enough to prove the statement for $r=2$. Let ${\rm rec}_{E}(\pi_{1})=(\rho_{1},N_{1})$ and ${\rm rec}_{E}(\pi_{2})=(\rho_{2},N_{2})$. Then we have,
\begin{equation*}
\begin{split}
{\rm rec}_{F}({\rm ai}_{E/F}(\pi_{1})\times {\rm ai}_{E/F}(\pi_{2})) & ={\rm rec}_{F}({\rm ai}_{E/F}(\pi_{1}))\oplus{\rm rec}_{F}({\rm ai}_{E/F}(\pi_{2}))\\
&= ({\rm ind}_{W_{E}}^{W_{F}}(\rho_{1}),\tilde{N_{1}})\oplus ({\rm ind}_{W_{E}}^{W_{F}}(\rho_{2}),\tilde{N_{2}})\\
&= ({\rm ind}_{W_{E}}^{W_{F}}(\rho_{1}\oplus\rho_{2}),\tilde{N_{1}}\oplus\tilde{N_{2}}).\\
\end{split}
\end{equation*}
By Lemma \ref{patnai} this is equal to ${\rm rec}_{F}({\rm ai}_{E/F}(\pi_{1}\times\pi_{2}))$ which finishes the proof of the lemma. 
 
\end{proof}

\begin{remark}
Henceforth in this article, every statement that we make for the base change setting has an automorphic induction analogue and vice versa. The proof in one setting is a verbatim translation of the proof in the other setting. To avoid repetition of arguments, from this point onwards we will give precise statements for both settings but prove only the one in the base change case. 
\end{remark}

\subsection{Compatibility with the Zelevinsky classification}\label{relzel}

\subsubsection{}\label{sss:3_not_1} 
The next lemma is a straightforward application of the local Langlands correspondence. We provide a proof for the sake of completeness. 

\begin{lemma}\label{strucpres}
\begin{enumerate}
\item Let $\sigma\in \mathcal A_{F}^{\circ}$ and let $\pi=\pi_{(\sigma)}\in\mathcal A_{F}$ be such that ${\rm Supp}\ \pi\subset \sigma^{\mathbb Z}$. Let ${\rm bc}_{E/F}(\sigma)=\sigma_{1}\times\cdots\times \sigma _{t}$ (see Lemma \ref{cuspfac}(1)). Then 
\begin{equation*}
{\rm bc}_{E/F}(\pi)= \pi_{(\sigma_{1})}\times \cdots \times \pi_{(\sigma_{t})}
\end{equation*}
(see \S \ref{def: cusp supp} for the notation).

\item Let $\sigma\in \mathcal A_{E}^{\circ}$ and let $\pi=\pi_{(\sigma)}\in\mathcal A_{E}$ be such that ${\rm Supp}\ \pi\subset \sigma^{\mathbb Z}$. Let ${\rm ai}_{E/F}(\sigma)=\sigma_{1}\times\cdots\times \sigma _{t}$ (see Lemma \ref{cuspfac}(2)). Then 
\begin{equation*}
{\rm ai}_{E/F}(\pi)= \pi_{(\sigma_{1})}\times \cdots \times \pi_{(\sigma_{t})}
\end{equation*}
(see \S \ref{def: cusp supp} for the notation).
\end{enumerate}
\end{lemma}
\begin{proof}
Let $\pi=L(\Delta_{1},\dots,\Delta_{s})$ and $\Pi= \pi_{(\sigma_{1})}\times \cdots \times \pi_{(\sigma_{t})}$. Note that, since $\sigma_{j}\neq \nu^{\alpha}\sigma_{j'}$ for any $\alpha\in \mathbb R$ if $j\neq j'$, the representation $\Pi$ is irreducible. We will show that ${\rm bc}_{E/F}(\pi)=\Pi$.

Let $\rec_{F}(\sigma)= \rho$ and $\rec_{E}(\sigma_{j})=\rho_{j}$ for $j=1,\dots,t$. Thus for any integer $r$ we have 
\begin{equation*}
(\nu_{F}^{'r}\rho)|_{W_{E}}=\nu_{E}^{'r}\rho_{1}\oplus\cdots\oplus\nu_{E}^{'r}\rho_{t}.
\end{equation*} 
Denote the representation space of $\nu_{F}^{'r}\rho$ by $V_{r}$ and that by $\nu^{'r}_{E}\rho_{j}$ by $V_{r,j}$. In other words, $V_{r}=\oplus_{j=1}^{t}V_{r,j}$ as a $W_{E}$-module. 

Set $\Delta_{i}=[\nu_{F}^{a_{i}}\sigma,\nu_{F}^{b_{i}}\sigma]$, $\Delta'_{i}=[\nu_{F}^{'a_{i}}\rho,\nu_{F}^{'b_{i}}\rho]$, $\Delta_{i,j}=[\nu_{E}^{a_{i}}\sigma_{j},\nu_{E}^{b_{i}}\sigma_{j}]$ and $\Delta'_{i,j}=[\nu_{E}^{'a_{i}}\rho_{j},\nu_{E}^{'b_{i}}\rho_{j}]$ (where $1\leq i \leq s$ and $1\leq j \leq t$). Then 
\begin{equation*}
{\rm rec}_{E}(\Pi)=\oplus_{j=1}^{t}(\oplus_{i=1}^{s}(\rho(\Delta'_{i,j}),N(\Delta'_{i,j}))).  
\end{equation*} 
Note that
\begin{equation*}
\begin{split}
{\rm res}_{E/F}({\rm rec}_{F}(\pi)) & = \oplus_{i=1}^{s}(\rho(\Delta'_{i})|_{W_{E}},N(\Delta'_{i}))\\
 & = \oplus_{i=1}^{s}((\oplus_{r=a_{i}}^{b_{i}}\nu^{'r}\rho)|_{W_{E}},N(\Delta'_{i}))\\
 & = \oplus_{i=1}^{s}((\oplus_{r=a_{i}}^{b_{i}}(\oplus_{j=1}^{t}\nu^{'r}\rho_{j})),N(\Delta'_{i}))\\
\end{split}
\end{equation*}

By the description of $N(\Delta'_{i})$ (provided in \S\ref{sss:classfn_wd_rep}), by rearranging the spaces $V_{r,j}$, we get that $N(\Delta'_{i})=\oplus_{j=1}^{t}N(\Delta'_{i,j})$ for every $i$. Thus we get that
\begin{equation*}
{\rm res}_{E/F}({\rm rec}_{F}(\pi))=\oplus_{i=1}^{s}(\oplus_{j=1}^{t}(\oplus_{r=a_{i}}^{b_{i}}\nu^{'r}\rho_{j}),\oplus_{j=1}^{t}N(\Delta'_{i,j})) = {\rm rec}_{E}(\Pi)
\end{equation*} 
and we obtain the first statement. 
\end{proof}

\subsubsection{Compatibility with the Zelevinsky involution} 
\begin{prop}\label{mainpr}
\begin{enumerate}
\item Let $\pi\in \mathcal A_{F}$ be a rigid representation. Then ${\rm bc}_{E/F}(\pi^{t})={\rm bc}_{E/F}(\pi)^{t}$. 
\item Let $\pi\in \mathcal A_{E}$ be a rigid representation. Then ${\rm ai}_{E/F}(\pi^{t})={\rm ai}_{E/F}(\pi)^{t}$. 
\end{enumerate}
\end{prop}
\begin{proof}
Suppose that ${\rm Supp}\ \pi\subset \sigma^{\mathbb Z}$ for some $\sigma\in \mathcal A_{F}^{\circ}$. Since the Zelevinsky involution of a representation preserves its cuspidal support, we have $\supp(\pi^{t})\subset \sigma^{\mathbb Z}$. It was shown in \cite{MW} that the action of the Zelevinsky involution on rigid representations is `oblivious' to the cuspidal line on which it is supported. In other words, $(\pi_{(\sigma_{i})})^{t}=(\pi^{t})_{(\sigma_{i})}$ for all $\sigma_{i}$. Using Lemma \ref{strucpres}(1) and the fact that Zelevinsky involution is a homomorphism of the Grothendieck ring of the general linear groups we get that  
\begin{equation*}
{\rm bc}_{E/F}(\pi^{t})=\prod_{i=1}^{t}(\pi^{t})_{(\sigma_{i})}=\prod_{i=1}^{t}(\pi_{(\sigma_{i})})^{t}={\rm bc}_{E/F}(\pi)^{t}.
\end{equation*}
\end{proof}

\section{Ladder representations}\label{s_lad}
The class of ladder representations was introduced in \cite{LM}. This class of irreducible representations has many interesting properties, for instance these are precisely the representations in the class of rigid representations whose Jacquet modules are semi-simple (see \cite[Corollary 4.11]{G}). Furthermore the Jacquet modules of a ladder representation are calculated explicitly in \cite[Corollary 2.2]{KL12}. Moreover, this class is preserved by the Zelevinsky involution and the algorithm provided in \cite{MW} to compute the Zelevinsky involution of an irreducible representation takes a much simpler form when the representation is a ladder (see \cite[\S 3]{LM}). Some of these structural properties make this class more approachable in comparison to the entire admissible dual for the purpose of distinction problems (for instance see \cite{MOS}). However the aforementioned properties will not play a direct role in this article. 
 
We will now recall the definition of ladder representations and collect some basic facts about them that we were going to use in this article. We will show that the rigid representations that are irreducibly induced from ladder representations remain in the class of representations irreducibly induced from ladders, under the two maps.
\subsection{Preliminaries on ladder and unitarizable representations} 
For now suppose $F$ to be an arbitrary non-archimedean local field.
\subsubsection{Definition of ladders and proper ladders}\label{sss: ladder}
\begin{definition} 
Let $\sigma\in \mathcal A_{F}^{\circ}$. Let the set $\mathfrak m=\{\Delta_1,\dots,\Delta_k\}$ be such that $\supp(\mathfrak m)\subset \sigma^{\mathbb Z}$ and write $\Delta_{i}=[\nu_{F}^{a_{i}}\sigma,\nu_{F}^{b_{i}}\sigma]$ ($a_{i},b_{i}\in \mathbb Z$). By renumbering the segments if required, we can assume that $a_{1}\geq\cdots \geq a_{k}$. Then $\mathfrak m$ is called a \emph{ladder} if 
\[
a_{1}>\cdots>a_{k}\ \ \  \text{and} \ \ \ b_{1}>\cdots>b_{k}. 
\]
It is called a \emph{proper ladder} if furthermore, $a_{i}\le b_{i+1}+1$ for all $i=1,\dots,k-1$. 
\end{definition}

\begin{definition}
\begin{enumerate}
\item A representation $\pi\in \aaa_{F}$ is called a ladder representation if $\pi=L(\mathfrak m)$ where $\mathfrak m$ is a ladder.
\item A representation $\pi\in \aaa_{F}$ is called a proper ladder representation if $\pi=L(\mathfrak m)$ where $\mathfrak m$ is a proper ladder.
\end{enumerate}
\end{definition}

\begin{example}
Let $\sigma\in \mathcal A_{F}^{\circ}$, $\mathfrak m_{1}=\{[2,3],[0,1]\}_{(\sigma)}$, and $\mathfrak m_{2}=\{[3,4],[0,1]\}_{(\sigma)}$. The multi-sets $\mathfrak m_{1}$ and $\mathfrak m_{2}$ are example of ladders of which only $\mathfrak m_{1}$ is a proper ladder.
\end{example}

Whenever we say that $\mult=\{\Delta_1,\dots,\Delta_k\}\in \sigma^{\mathbb Z}$ is a ladder or a proper ladder, we implicitly assume that $\mult$ is already ordered as in the definition above, namely so that $a_{1}>\dots > a_{k}$ where $b(\Delta_{i})=\nu^{a_{i}}\sigma$.

We will denote the subset of ladder representations of $\aaa_{F}$ by $\mathcal L=\mathcal L_{F}$ and the subset of proper ladders by $\mathcal L_{p}=\mathcal L_{p,F}$. The class of representations irreducibly induced from ladders will be denoted by $\mathcal L_{{\rm ind}}=\mathcal L_{{\rm ind},F}$.

\subsubsection{}
The next proposition follows directly from \cite[Theorem 16]{LM}.
\begin{prop}\label{prop: prop}
Let $\pi\in \mathcal L_{{\rm ind}}$. Then $\pi$ can be written as $\pi_{1}\times \cdots \times \pi_{k}$ where each $\pi_{i}$ is a proper ladder. The decomposition is unique up to a reordering of the $\pi_{i}$.
\end{prop}

\subsubsection{Tadi{\'c}'s classification of unitarizable representations}\label{sss:tadic}
An important example of a proper ladder representation is when $a_{i}=a_{i+1}+1$ and $b_{i}=b_{i+1}+1$ for $i=1,\dots,k-1$. Define a Speh representation to be a proper ladder such that the underlying multi-set satisfies this property. Notice that we are not assuming that Speh representations are unitarizable in general. We will use the term {\it unitarizable Speh} in this paper for a Speh representation that lies in $\aaa_{F}^{u}$. 

For a unitarizable Speh representation $\tau$ and a real number $\alpha\in (-\frac{1}{2},\frac{1}{2})$, define $\pi(\tau,\alpha)$ to be the representation $\nu^{\alpha}\tau\times \nu^{-\alpha}\tau$. By \cite[Proposition 8.5]{Z} it is irreducible. We now recall the classification of the unitarizable representations of general linear groups (see \cite[Theorem D]{Ta}).  
\begin{theorem}\label{thm:tadic}
\begin{enumerate}
\item The representations $\pi(\tau,\alpha)$ lie in $\aaa_{F}^{u}$. 
\item Every representation $\pi\in\aaa_{F}^{u}$ can be written as 
\begin{equation*}
\pi=\pi_{1}\times \cdots \times \pi_{t} 
\end{equation*}
where each $\pi_{i}$ is either a unitarizable Speh representations or a representation of the form $\pi(\tau,\alpha)$. 
\end{enumerate}
\end{theorem}
In particular, $\aaa_{F}^{u}\subset \mathcal L_{{\rm ind}}$.  

\subsection{Base change, automorphic induction and ladder representations}
We return to the case when $E/F$ is a cyclic extension of characteristic different than two such that $d=[E:F]$ is prime.
\subsubsection{}
\begin{prop}\label{unlad}
\begin{enumerate}
\item Let $\pi\in \mathcal L_{{\rm ind},F}$ be a rigid representation. Then ${\rm bc}_{E/F}(\pi)\in \mathcal L_{{\rm ind},E}$. 
\item Let $\pi\in \mathcal L_{{\rm ind},E}$ be a rigid representation. Then ${\rm ai}_{E/F}(\pi)\in \mathcal L_{{\rm ind},F}$. 
\end{enumerate}
\end{prop}
\begin{proof}
The result is obtained by a direct application of  Lemma \ref{cuspfac} and Lemma \ref{strucpres}.
\end{proof}

\subsubsection{} The necessity of the `rigidity' hypothesis in Proposition \ref{unlad} is shown by the following example. Let $\pi=L([\nu_{F}^{2},\nu_{F}^{3}],[1,\nu_{F}^{2}],[\kappa_{E/F}])$. It is easy to see that ${\rm bc}_{E/F}(\pi)=L([\nu_{E}^{2},\nu_{E}^{3}],[1,\nu_{E}^{2}],[\nu_{E}])$ is not in $\mathcal L_{{\rm ind},E}$. One can easily construct similar examples to demonstrate the failure of the statement without rigidity in the case of automorphic induction as well.

\subsubsection{}The hypothesis of rigidity can be removed from the above statements if we further assume that the representations we are dealing with are unitarizable.  

\begin{prop}\label{untri}
\begin{enumerate}
\item Let $\pi\in \mathcal A_{F}^{u}$. Then ${\rm bc}_{E/F}(\pi)\in \mathcal A_{E}^{u}$. 
\item Let $\pi\in \mathcal A_{E}^{u}$. Then ${\rm ai}_{E/F}(\pi)\in \mathcal A_{F}^{u}$.
\end{enumerate}
\end{prop}
\begin{proof}
Suppose $\pi\in\aaa_{F}^{u}$. Write $\pi=\pi_{1}\times \cdots\times \pi_{k}$ such that each $\pi_{i}$ is either a unitarizable Speh representation or a representation of the form $\pi(\tau_{i},\alpha_{i})$ for some unitarizable Speh $\tau_{i}$ and some $\alpha_{i}\in (-\frac{1}{2},\frac{1}{2})$. If $\pi_{i}$ is a unitarizable Speh then it is clear by Lemma \ref{cuspfac}, and Lemma \ref {strucpres} that ${\rm bc}_{E/F}(\pi_{i})\in \mathcal A_{E}^{u}$. Suppose that $\pi_{i}=\nu^{\alpha_{i}}\tau_{i}\times \nu^{-\alpha_{i}}\tau_{i}$. By  Lemma \ref{cuspfac}, and Lemma \ref {strucpres}, $\nu^{\alpha_{i}}{\rm bc}_{E/F}(\tau_{i})\times \nu^{-\alpha_{i}}{\rm bc}_{E/F}(\tau_{i})$ is a product of Speh representations supported on different cuspidal lines, and is thus irreducible. Hence by Lemma \ref{bcprod}(1), ${\rm bc}_{E/F}(\pi_{i})=\nu^{\alpha_{i}}{\rm bc}_{E/F}(\tau_{i})\times \nu^{-\alpha_{i}}{\rm bc}_{E/F}(\tau_{i})$ which is again unitarizable by Theorem \ref{thm:tadic}. Thus we have ${\rm bc}_{E/F}(\pi_{i})\in \mathcal A_{E}^{u}$ for each $i$. 

Since a representation induced from unitarizable representations is irreducible, appealing to Lemma \ref{bcprod}(1), we get that ${\rm bc}_{E/F}(\pi)={\rm bc}_{E/F}(\pi_{1})\times \cdots\times {\rm bc}_{E/F}(\pi_{k})$. Since the induced representation is also unitarizable, this proves (1). 
\end{proof}

\section{Degenerate Whittaker models}\label{relsl}
We now study the degenerate Whittaker models and their relationships with the two maps. 
\subsection{Definition of degenerate Whittaker models}
We briefly recall the definition of degenerate Whittaker models as provided in \cite[\S 8.3]{Z}. Given a composition ${\bf d}=(\lambda_{1},\dots,\lambda_{l})$ of $n$ ordered such that $\lambda_{1}\geq \cdots\geq \lambda_{l}$, define the character $\theta=\theta_{{\bf d}}$ of $U_{n}$ by $\theta((u_{i,j}))=\psi(\sum u_{i,i+1})$ where $i$ runs over $1,\dots,n-1$ except
\begin{equation*}
n-\lambda_{1}, n-(\lambda_{1}+\lambda_{2}),\dots, n-(\lambda_{1}+\cdots+\lambda_{l-1}).
\end{equation*}
(See \S\ref{sss:gen} for the definition of $U_{n}$ and $\psi$.) Say that a representation $\pi\in \aaa_{F}(n)$ has a degenerate Whittaker model with respect to the sequence ${\bf d}$ if $\Hom_{U_{n}}(\pi, \theta_{{\bf d}})\neq 0$.

It was shown in \cite[Corollary 8.3]{Z} that every $\pi\in \aaa_{F}$ has a degenerate Whittaker model. 

\subsection{The depth sequence and $SL(2)$-type of an irreducible representation}
For $\pi\in \Alg({\rm GL}_{n}(F))$ and any $r=0,\dots,n$ we denote by $\pi^{(r)}$ the $r$-th derivative of $\pi$ as defined in \cite[\S 3.5 and \S 4.3]{BZ1}. It is a functor from $\Alg({\rm GL}_{n}(F))$ to $\Alg({\rm GL}_{n-r}(F))$. If the integer $r$ is such that $\pi^{(r)}\neq 0$ and $\pi^{(r+k)}=0$ for any $k\in \mathbb Z_{>0}$, then we call the representation $\pi^{(r)}$ the {\it highest derivative} of $\pi$ and the integer $r$ the {\it depth} of $\pi$.

\subsubsection{Definition of a depth sequence}\label{sss:def_depth} Given $\pi\in \aaa_{F}(n)$, we recursively define the irreducible representations $\tau_{0},\tau_{1},\dots,\tau_{l}$ and an integer sequence ${\bf d}(\pi)=(\lambda_{1},\dots, \lambda_{l})$ such that $\tau_{0}=\pi$, $\tau_{l}$ is the trivial representation of the trivial group, and $\tau_{i+1}:=\tau_{i}^{(\lambda_{i+1})}$ is the highest derivative of $\tau_{i}$ ($i=0,\dots, l-1$). We call this sequence the {\it depth sequence} of the irreducible representation $\pi$. Clearly $\lambda_{1}+\cdots+\lambda_{l}=n$ and by \cite[Theorem 8.1]{Z} we get that $\lambda_{1}\geq \cdots\geq \lambda_{l}$. Thus any depth sequence of an element of $\aaa_{F}(n)$ can be identified with an element of $\mathcal P(n)$. 

\subsubsection{}We now recall \cite[Corollary 8.3]{Z}.  
\begin{theorem}
Every $\pi\in \aaa_{F}$ has a degenerate Whittaker model with respect to its depth sequence ${\bf d}(\pi)$ with multiplicity one.
\end{theorem}

\subsubsection{Definition of an $SL(2)$-type} The $SL(2)$-type of an irreducible representation was first defined in \cite[Definition 1]{V} for the ones in the unitarizable dual. The definition was then extended to the admissible dual in \cite[Remark 2]{OS}. We recall it below.
\begin{defn}\label{defn: sl2}
Let $\pi\in\mathcal A_{F}$. Then the $SL(2)$-type of $\pi$ is defined to be the partition $P_{F}({\rm rec}_{F}(\pi^{t}))$ where $P_{F}$ is the map defined in \S\ref{sss:def_part}. It is denoted by $\mathcal V(\pi)$.
\end{defn}

\subsubsection{Relation between the two partitions} Given $\pi\in \aaa_{F}$, we will think of $\mathcal V(\pi)$ as a composition by ordering the elements of this partition in a non-increasing manner. For a composition $f$, denote by $f^{t}$ its conjugate composition. We have the following non-archimedean analogue of \cite[Theorem 2.4.2]{GOSS}:
\begin{lemma}\label{sl_depth}
For $\pi\in \mathcal A_{F}$. Then $\mathcal V(\pi)={\bf d}(\pi)^{t}$. 
\end{lemma}
\begin{proof}
Let $\pi=Z(\mathfrak m)$. The statement for a rigid multi-set $\mathfrak m$ follows directly from Lemma \ref{lem_def_part} and \cite[Theorem 8.1]{Z}. For an arbitrary $\pi\in \mathcal A_{F}$, write $\pi=\pi_{1}\times \cdots \times \pi_{t}$ such that the $\pi_{i}$ are rigid and supported on cuspidal lines that are pairwise disjoint. By \cite[Proposition 8.5]{Z} we get that $\mathcal V(\pi)=\mathcal V(\pi_{1})+\cdots+\mathcal V(\pi_{t})$. By adding 0's at the end if necessary, assume that each composition $\mathcal V(\pi_{i})^{t}$ is of the same length. For any two compositions $f_{1}$ and $f_{2}$ of $n_{1}$ and $n_{2}$ respectively of same length, denote by $f_{1}+_{c}f_{2}$ the composition of $n_{1}+n_{2}$ given by coordinate wise addition. Since $(f_{1}+f_{2})^{t}= f_{1}^{t}+_{c}f_{2}^{t}$, we obtain 
\begin{equation*}
\mathcal V(\pi)^{t}={\bf d}(\pi_{1})+_{c}\cdots+_{c}{\bf d}(\pi_{t})= {\bf d}(\pi).
\end{equation*} 
\end{proof}

\subsection{Degenerate Whittaker models and the two maps}\label{ss:deg_whit}
We begin by studying the $SL(2)$-type of the base change (or the automorphic induction) lift of an irreducible representation. 

\begin{theorem}\label{mainth}
\begin{enumerate}
\item Let $\pi\in \aaa_{F}$ be a rigid representation. Then $\mathcal V(\pi)=\mathcal V({\rm bc}_{E/F}(\pi))$.
\item  Let $\pi\in \aaa_{E}$ be a rigid representation. Then $d\mathcal V(\pi)=\mathcal V({\rm ai}_{E/F}(\pi))$.
\end{enumerate}
\end{theorem}
\begin{proof}
It is clear from the definition of base change that for $\pi\in \mathcal A_{F}$ that $P_{F}({\rm rec}_{F}(\pi))=P_{E}({\rm rec}_{E}({\rm bc}_{E/F}(\pi)))$. The result now follows from Proposition \ref{mainpr}.
\end{proof}

The hypothesis of rigidity in Theorem \ref{mainth} is essential as demonstrated by the following example. Take $\pi=L([1,\nu_{F}],[\kappa\nu_{F}^{2},\kappa\nu_{F}^{3}])$. Then ${\rm bc}_{E/F}(\pi)=L([1,\nu_{E}],[\nu_{E}^{2},\nu_{E}^{3}])$ and by Lemma \ref{lem_def_part}) we get that $\mathcal V(\pi)\neq \mathcal V({\rm bc}_{E/F}(\pi))$. A similar example can be constructed in the case of automorphic induction.   

However, as earlier, the rigidity hypothesis can be removed if we assume that $\pi$ is unitarizable. 

\begin{theorem}\label{mainth_2}
\begin{enumerate}
\item Let $\pi\in\aaa_{F}^{u}$. Then $\mathcal V(\pi)=\mathcal V({\rm bc}_{E/F}(\pi))$.
\item Let $\pi\in\aaa_{E}^{u}$. Then $d\mathcal V(\pi)=\mathcal V({\rm ai}_{E/F}(\pi))$.
\end{enumerate}
\end{theorem}
\begin{proof}
The result is obtained by applying Theorem \ref{mainth} to the class of Speh representations and arguing as in the proof of Proposition \ref{untri}.
\end{proof}

Finally we have the following result showcasing the behavior of the two maps with respect to degenerate Whittaker models.
\begin{theorem}\label{deg_whit_func}
\begin{enumerate}
\item Let $\pi\in \aaa_{F}$ be either rigid or unitarizable. Then ${\rm bc}_{E/F}(\pi)$ has a degenerate Whittaker model given by the sequence ${\bf d}(\pi)$. 
\item Let $\pi\in \aaa_{E}$ be either rigid or unitarizable. Then ${\rm ai}_{E/F}(\pi)$ has a degenerate Whittaker model given by the sequence $\underbrace{\operatorname{{\bf d}(\pi)+_{c}\cdots+_{c}{\bf d}(\pi)}}\limits_{d-\text{times}}$.
\end{enumerate}
\end{theorem}
\begin{proof}
The result is an immediate consequence of Theorem \ref{mainth}, Theorem \ref{mainth_2} and Lemma \ref{sl_depth}.
\end{proof}

\section{Klyachko models}\label{relkts}
We begin this section by recalling the definition of Klyachko models and the classification results for ladder representations with respect to these models that were obtained in \cite{MOS}. We use these results then to show that the two maps preserve the Klyachko type of ladder representations in an appropriate sense. 
\subsection{Definition of Klyachko types}
\subsubsection{}\label{defkt}
For a decomposition $n=2k+r$ let
\begin{equation*}
H_{2k,r}=\{\begin{pmatrix} h & X \\ 0 & u \end{pmatrix}: h\in \Sp_{2k}(F),\,X\in M_{2k\times r}(F),\,u\in U_r\}
\end{equation*}
and $\psi=\psi_{2k,r}$ be defined by
\begin{equation*}
\psi(\begin{pmatrix} h & X \\ 0 & u \end{pmatrix})=\psi_{r}(u).
\end{equation*}
(See \S \ref{sss:gen} for the definition of $U_r$ and its character $\psi_{r}$.)

\begin{defn}
Let $\pi\in\aaa_{F}(n)$. If $\pi$ is $(H_{2k,r},\psi)$-distinguished for some decomposition $n=2k+r$ then we say that it admits a Klyachko model of type $r$. In this case the integer $r$ is referred to as the Klyachko type of $\pi$ and denoted by $r(\pi)$. 
\end{defn}
We remark that $r(\pi)$ is well defined as by \cite[Theorem 1]{OS2} the Klyachko type of an irreducible representation is unique if it exists.  

\subsubsection{}
\begin{remark}
Note that for any $\pi\in\Pi(\GL_{n}(F))$, being $(H_{2k,r},\psi)$-distinguished is independent of the choice of non-trivial character $\psi$ of $F$. Indeed, for any other character $\psi'\ne 1$ there is a diagonal matrix $a\in \GL_{n}(F)$ normalizing $H_{2k,r}$ such that $\psi'_{2k,r}(h)=\psi_{2k,r}(ah a^{-1})$ for all $h\in H_{2k,r}$.
\end{remark}

\subsection{The classification}
We now recall the classification of ladder representations with respect to the Klyachko models.
\subsubsection{Right aligned segments} We define the following relation on segments of cuspidal representations.
\begin{definition}\label{def: ra seg}
For segments $\Delta=[\nu^{a}\sigma,\nu^{b}\sigma]$ and $\Delta'=[\nu^{a'}\sigma,\nu^{b'}\sigma]$ ($a,a',b,b'\in \mathbb Z$) we say that $\Delta'$ is right-aligned with $\Delta$ and write $\Delta' \vdash \Delta$ if 
\begin{itemize}
\item $a\ge a'+1$ and
\item $b= b'+1$.
\end{itemize}
We label this relation by the integer $r=s(a-a'-1)$ where $\sigma\in \aaa_{F}^{\circ}(s)$ and write $\Delta'\vdash_r \Delta$.
\end{definition}
Note, in particular, that $\Delta'\vdash_0\Delta$ means that $\Delta=\nu\Delta'$.

\subsubsection{} We now provide a description of the ladder representations that admit any particular Klyachko model (\cite[Proposition 14.5 and Theorem 14.7]{MOS}).
\begin{theorem}\label{thm: kly lad}
\begin{enumerate}
\item Let $\mult=\{\Delta_1,\dots,\Delta_t\}$ be a proper ladder, so that $L(\mult)\in \mathcal L_{p}\cap \aaa_{F}(n)$ and let $n=2k+r$. If $t$ is odd, let $s$ be such that $L(\Delta_1)\in \aaa_{F}(s)$, otherwise, set $s=0$. Then $L(\mult)$ is $(H_{2k,r},\psi)$-distinguished if and only if $\Delta_{t-2i}\vdash_{r_i}\Delta_{t-2i-1}$ for some $r_i$, $i=0,\dots,\lfloor t/2\rfloor-1$ and $r=r_0+\cdots +r_{\lfloor t/2\rfloor-1}+s$. 
\item Let $\pi$ be a ladder representation and assume that $\pi=\pi_1\times\cdots\times\pi_l$ is the unique decomposition of $\pi$ as a product of proper ladder representations (see Proposition \ref{prop: prop}). Then $\pi$ admits a Klyachko model if and only if $\pi_i$ admits a Klyachko model for all $i=1,\dots,l$. Furthermore, in that case $r(\pi)=r(\pi_1)+\cdots+r(\pi_l)$.
\end{enumerate}
\end{theorem}

\subsection{Relationship with the two maps}

We now prove Theorem \ref{Klyachkoai_3} in the introduction. The analogous result for unitarizable representations and for the case of the base change map (Corollary \ref{Klyachkoai_uni} (1) below) was obtained in \cite[Corollary 6.1]{OS}. This was done there by observing that $r(\pi)=\sum_{i=0}^{\infty}\mathcal V(\pi)(2i+1)$ (using Lemma \ref{lem_def_part}) and then using the fact that $SL(2)$-type of a unitarizable representation is preserved by base change. Unlike the unitarizable representations though, a ladder representation may not have a Klyachko model and it is not a priori clear if for a representation $\pi$ admitting a Klyachko model, even ${\rm bc}_{E/F}(\pi)$ will admit one. However Theorem \ref{thm: kly lad} allows us to determine precisely which ladders have a Klyachko model and enables us to prove  Theorem \ref{Klyachkoai_3} and Corollary \ref{Klyachkoai_uni} directly, without resorting to $SL(2)$-types, as we see below.
   
\subsubsection{}
\begin{theorem}\label{Klyachkoai}
\begin{enumerate}
\item Let $\pi\in \mathcal L_{F}$. Then $\pi$ admits a Klyachko model if and only if ${\rm bc}_{E/F}(\pi)$ admits one. Moreover $r({\rm bc}_{E/F}(\pi))=r(\pi)$. 
\item Let $\pi\in \mathcal L_{E}$. Then $\pi$ admits a Klyachko model if and only if ${\rm ai}_{E/F}(\pi)$ admits one. Moreover $ r({\rm ai}_{E/F}(\pi))=dr(\pi)$. 
\end{enumerate}
\end{theorem}
\begin{proof}
Let $\sigma\in \aaa_{F}^{\circ}(m)$ be such that $\supp(\pi)\in \sigma^{\mathbb Z}$. Using Lemma \ref{cuspfac}, we write ${\rm bc}_{E/F}(\sigma)=\sigma_{1}\times \cdots \times \sigma_{t}$ where $\sigma_{i}\in \aaa_{E}^{\circ}(\frac{m}{t})$. By Lemma \ref{strucpres}, we get that ${\rm bc}_{E/F}(\pi)=\pi_{(\sigma_{1})}\times \cdots \times \pi_{(\sigma_{t})}$ (see \S \ref{def: cusp supp} for the notation). 

If $\pi$ has a Klyachko model of type $r(\pi)$, then by Theorem \ref{thm: kly lad} each $\pi_{(\sigma_{i})}$ has a Klyachko model and $r(\pi_{(\sigma_{i})})=\frac{r(\pi)}{t}$. By the hereditary property of Klyachko model (\cite[Proposition 13.3]{MOS}), ${\rm bc}_{E/F}(\pi)$ has the Klyachko model of type $r(\pi)$. This gives us the `only if' part of (1).

For the `if' part, note that the cuspidal lines of $\sigma_{i}$ and $\sigma_{j}$ are pairwise disjoint if $i\neq j$ and thus, by \cite[Proposition 13.4]{MOS}, each $\pi_{(\sigma_{i})}$ admits a Klyachko model. Appealing to Theorem \ref{thm: kly lad} again, we get that $\pi$ admits a Klyachko model.
\end{proof}

\subsubsection{}
\begin{corr}\label{Klyachkoai_cor}
\begin{enumerate}
\item Let $\pi_{i}\in \mathcal L_{F}$ ($i=1,\dots,k$) be such that $\pi:=\pi_{1}\times \cdots \times \pi_{k}\in \mathcal L_{{\rm ind}}$. Moreover assume that each $\pi_{i}$ admits a Klyachko model and that $\pi$ is rigid. Then ${\rm bc}_{E/F}(\pi)$ admits a Klyachko model with $r({\rm bc}_{E/F}(\pi))=\sum_{i=1}^{k}r(\pi_{i})$.
\item  Let $\pi_{i}\in \mathcal L_{E}$ ($i=1,\dots,k$) be such that $\pi:=\pi_{1}\times \cdots \times \pi_{k}\in \mathcal L_{{\rm ind}}$. Moreover assume that each $\pi_{i}$ admits a Klyachko model and that $\pi$ is rigid. Then ${\rm ai}_{E/F}(\pi)$ admits a Klyachko model with $r({\rm ai}_{E/F}(\pi))=d(\sum_{i=1}^{k}r(\pi_{i}))$.
\end{enumerate}
\end{corr}
\begin{proof}
By \cite[Lemma 5.17, Proposition 5.20, Lemma 5.21]{LM1}, for any $\sigma'\in \aaa_{E}^{\circ}$, the representation $(\pi_{1})_{(\sigma')}\times \cdots \times (\pi_{k})_{(\sigma')}$ is irreducible (see \S \ref{def: cusp supp} for the notation) and is thus equal to $\pi_{(\sigma')}$. Therefore, by Lemma \ref{cuspfac} and Lemma \ref{strucpres}, we have 
\begin{equation*}
{\rm bc}_{E/F}(\pi)={\rm bc}_{E/F}(\pi_{1})\times \cdots\times{\rm bc}_{E/F}(\pi_{k}).
\end{equation*}
Theorem \ref{Klyachkoai} along with the hereditary property of Klyachko models (\cite[Proposition 13.3]{MOS}) now gives us the corollary.
\end{proof}

\subsubsection{Symplectic models}
Let $\pi\in \aaa_{F}(2n)$. Then $\pi$ is said to have a symplectic model if it admits a Klyachko model of type 0. In other words, the representation $\pi$ has to be $\Sp_{2n}(F)$-distinguished.  

Recall that any $\pi \in \mathcal L_{{\rm ind}}$ can be written as $\pi=\pi_{1}\times \cdots \times \pi_{k}$ where $\pi_{i}\in \mathcal L_{p}$ ($i=1,\dots,k$) and the multi-set consisting of the representations $\pi_{i}$ is uniquely determined by $\pi$ (see Proposition \ref{prop: prop}). We consider the following property for representations in $\mathcal L_{{\rm ind}}$.
\begin{hypothesis}\label{symp_claim} 
Let $\pi \in \mathcal L_{{\rm ind}}$ and $\pi_{i}\in \mathcal L_{p}$ ($i=1,\dots,k$) be such that $\pi=\pi_{1}\times \cdots \times \pi_{k}$. Then $r(\pi)=0$ implies $r(\pi_{i})=0$ for every $i$. 
\end{hypothesis}
Hypothesis \ref{symp_claim} was proved in \cite{MOS} for all representations in $\mathcal L_{{\rm ind}}$ that satisfy a combinatorial condition (\cite[Proposition 12.5]{MOS}) on the underlying multi-set. 

Assuming that every representation in $\mathcal L_{{\rm ind}}$ satisfies Hypothesis \ref{symp_claim}, we can improve Corollary \ref{Klyachkoai_cor} for the special case of symplectic  models in the following way.
\begin{corr}\label{symp_cor} 
Suppose that Hypothesis \ref {symp_claim} hold for every representation in $\mathcal L_{{\rm ind},F}$ and $\mathcal L_{{\rm ind},E}$. 
\begin{enumerate}
\item Let $\pi\in \mathcal L_{{\rm ind},F}$ and $\pi_{i}\in \mathcal L_{p,F}$ ($i=1,\dots,k$) be such that $\pi=\pi_{1}\times \cdots \times \pi_{k}$. Suppose further that $\pi$ is rigid. Then $\pi$ has a symplectic model if and only if ${\rm bc}_{E/F}(\pi)$ has a symplectic model.
\item Let $\pi\in \mathcal L_{{\rm ind},E}$ and $\pi_{i}\in \mathcal L_{p,E}$ ($i=1,\dots,k$) be such that $\pi=\pi_{1}\times \cdots \times \pi_{k}$. Suppose further that $\pi$ is rigid. Then $\pi$ has a symplectic model if and only if ${\rm ai}_{E/F}(\pi)$ has a symplectic model.
\end{enumerate}
\end{corr}
\begin{proof}
The `only if' part follows directly from the assumption that $\pi$ satisfies Hypothesis \ref{symp_claim} and Corollary \ref{Klyachkoai_cor}. The `if' part follows from Lemma \ref{cuspfac}, Lemma \ref{strucpres}, \cite[Corollary 5.3 and Lemma 5.9]{MOS}, Theorem \ref{thm: kly lad} and our assumption that the Hypothesis \ref {symp_claim} hold for every representation in $\mathcal L_{{\rm ind}}$. 
\end{proof}

\subsubsection{} The hypothesis of rigidity is essential for both Corollary \ref{Klyachkoai_cor} and Corollary \ref{symp_cor} to hold. For example let $\pi_{1}=\kappa$, $\pi_{2}=L([\nu_{F},\nu_{F}^{2}],[\nu_{F}^{2},\nu_{F}^{3}])$ and $\pi=\pi_{1}\times \pi_{2}$. By Theorem \ref{thm: kly lad} and the hereditary property of Klyachko models (\cite[Proposition 13.3]{MOS}), the irreducible representation $\pi$ has a Klyachko model. On the other hand, Theorem \ref{thm: kly lad} implies that ${\rm bc}_{E/F}(\pi)$ doesn't have any Klyachko model. Similar counterexamples can be constructed for the case of automorphic induction (and to Corollary \ref{symp_cor}) if $\pi$ is not assumed to be rigid.

\subsubsection{Unitarizable case}However the hypothesis of rigidity can be removed for the case of unitarizable representations. Let $\pi$ be unitarizable. Write $\pi=\pi_{1}\times \cdots\times \pi_{k}$ such that each $\pi_{i}$ is either a unitarizable Speh representation or a representation of the form $\nu^{\alpha}\tau\times \nu^{-\alpha}\tau$, where $\alpha\in (-\frac{1}{2},\frac{1}{2})$ and $\tau$ is a unitarizable Speh representation. Theorem \ref{thm: kly lad} gives that each Speh representation admits a Klyachko model and thus by \cite[Proposition 13.3]{MOS}, so does $\pi$ with its Klyachko type equal to $\sum_{i=1}^{k}r(\pi_{i})$. Thus by Proposition \ref{untri} if $\pi\in \aaa_{F}^{u}$ (resp., $\pi\in \aaa_{E}^{u}$) then both $\pi$ and ${\rm bc}_{E/F}(\pi)$ (resp., ${\rm ai}_{E/F}(\pi)$) have a Klyachko model. We further have the following:
\begin{corr}\label{Klyachkoai_uni}
\begin{enumerate}
\item Let $\pi\in \aaa_{F}^{u}$. Then $r({\rm bc}_{E/F}(\pi))=r(\pi)$. 
\item Let $\pi\in \aaa_{E}^{u}$. Then $r({\rm ai}_{E/F}(\pi))=dr(\pi)$. 
\end{enumerate}
\end{corr}
\begin{proof}
Let $\pi\in \aaa^{u}_{F}$ and $\pi=\pi_{1}\times\dots\times\pi_{k}$ as above. Arguing as in the proof of Proposition \ref{untri}, we get that ${\rm bc}_{E/F}(\pi)={\rm bc}_{E/F}(\pi_{1})\times \cdots \times {\rm bc}_{E/F}(\pi_{k})$ where each ${\rm bc}_{E/F}(\pi_{i})\in \aaa_{E}^{u}$ and admits a Klyachko model. If $\pi_{i}$ is a unitarizable Speh representation, then by Theorem \ref{Klyachkoai} $r(\pi_{i})=r({\rm bc}_{E/F}(\pi_{i}))$. Suppose now that $\pi_{i}=\nu^{\alpha}\tau\times \nu^{-\alpha}\tau$ for $\alpha\in (-\frac{1}{2},\frac{1}{2})$ and a unitarizable Speh representation $\tau$. As in the proof of Proposition \ref{untri}, we have ${\rm bc}_{E/F}(\pi_{i})={\rm bc}_{E/F}(\nu^{-\alpha}\tau)\times {\rm bc}_{E/F}(\nu^{-\alpha}\tau)$. By Theorem \ref{Klyachkoai}, we have $r({\rm bc}_{E/F}(\nu^{-\alpha}\tau))=r(\nu^{-\alpha}\tau)$ and $r({\rm bc}_{E/F}(\nu^{\alpha}\tau))=r(\nu^{\alpha}\tau)$. Therefore even in this case
\begin{equation*}
r({\rm bc}_{E/F}(\pi_{i}))=r({\rm bc}_{E/F}(\nu^{-\alpha}\tau))+r( {\rm bc}_{E/F}(\nu^{-\alpha}\tau))=r(\nu^{-\alpha}\tau)+r(\nu^{\alpha}\tau)=r(\pi_{i}).
\end{equation*}
Thus we have $r({\rm bc}_{E/F}(\pi_{i}))=r(\pi_{i})$ for every $i=1,\dots,k$.

By the hereditary property of Klyachko models (\cite[Proposition 13.3]{MOS}) we get
\begin{equation*}
r({\rm bc}_{E/F}(\pi))=\sum_{i=1}^{k}r({\rm bc}_{E/F}(\pi_{i}))=\sum_{i=1}^{k}r(\pi_{i})=r(\pi).
\end{equation*}
\end{proof}

\section{Fiber under the two maps}\label{s:fiber}

We now investigate the fibers of the base change and automorphic induction maps. We begin by explicitly describing the fiber of an arbitrary rigid representation in the image.

\subsection{Description of the fiber under a rigid representation}
\subsubsection{}
\begin{lemma}\label{fiber:maps}
\begin{enumerate}
\item Suppose that $\Pi=L(\mathfrak m)$ be such that $\supp (\Pi)\subset {\sigma}^{\mathbb Z}$ for some $\sigma\in \aaa_{E}^{\circ}$. Let $\Pi$ be in the image of the map ${\rm bc}_{E/F}$. Let $\kappa$ be a character of $F^{\times}$ with kernel equal to ${\rm N}_{E/F}(E^{\times})$. Then there exists $\sigma'\in \aaa_{F}^{\circ}$ such that ${\rm bc}_{E/F}(\sigma')=\sigma$ and the fiber ${\rm bc}_{E/F}^{-1}(\Pi)$ consists of all the representations of the form 
\begin{equation}\label{eq_fiber:maps1}
L(\mathfrak m_{1})\times \kappa L(\mathfrak m_{2})\times \cdots \times\kappa^{d-1} L(\mathfrak m_{d})
\end{equation}
where the multi-sets $\mathfrak m_{i}$ are such that each $\mathfrak m_{i}\subset {\sigma'}^{\mathbb Z}$ and $\mathfrak m_{1}+\cdots+\mathfrak m_{d}=\mathfrak m_{(\sigma')}$ (see \S \ref{def: cusp supp} for the notation). (Some of the $\mathfrak m_{i}$ can possibly be empty.)

\item Suppose that $\Pi=L(\mathfrak m)$ be such that $\supp (\Pi)\subset {\sigma}^{\mathbb Z}$ for some $\sigma\in \aaa_{F}^{\circ}$. Let $\Pi$ be in the image of the map ${\rm ai}_{E/F}$. Let $\gamma$ be a fixed non-trivial element of  ${\rm Gal}(E/F)$. Then there exists $\sigma'\in \aaa_{E}^{\circ}$ such that ${\rm ai}_{E/F}(\sigma')=\sigma$ and the fiber ${\rm ai}_{E/F}^{-1}(\Pi)$ consists of all the representations of the form 
\begin{equation*}
L(\mathfrak m_{1})\times L(\mathfrak m_{2})\times \cdots \times L(\mathfrak m_{d})
\end{equation*}
where the multi-sets $\mathfrak m_{i}$ are such that $\mathfrak m_{i}\subset ((\sigma')^{\gamma^{i}})^{\mathbb Z}$ for each $i$ and $(\mathfrak m_{1})_{(\sigma)}+\cdots+(\mathfrak m_{d})_{(\sigma)}=\mathfrak m$ (see \S \ref{def: cusp supp} for the notation). (Some of the $\mathfrak m_{i}$ can possibly be empty.)
\end{enumerate}
\end{lemma}
\begin{proof}
Note first that if $\sigma$ lies in the image of the base change map, then by the local Langlands correspondence there exists a $\sigma'\in \aaa_{F}^{\circ}$ such that $\rec(\sigma')|_{W_{E}}=\rec(\sigma)$. Fix such a $\sigma'$. In this case, by Clifford theory, we have 
\begin{equation}\label{eq_fiber:maps2}
{\rm bc}_{E/F}^{-1}(\sigma)=\{\kappa^{i}\sigma'\mid i=0,\dots,d-1\}.
\end{equation}
Let $\pi\in {\rm bc}_{E/F}^{-1}(\Pi)$. Write $\Pi=L(\Delta_{1},\dots,\Delta_{t})$ and $\pi=L(\Delta_{1}',\dots, \Delta_{t'}')$. Thus ${\rm rec}(\pi)=\oplus_{i=1}^{t'}(\rho(\Delta'_{i}), N(\Delta_{i}'))$. If there exists a $\sigma''\in \supp (\pi)$ such that ${\rm bc}_{E/F}(\sigma'')$ is not cuspidal, then we get a contradiction to the rigidity of $\Pi$ (using Lemma \ref{cuspfac}). In other words, $(\rho(\Delta'_{i})|_{W_{E}}, N(\Delta_{i}'))$ is an indecomposable Weil-Deligne representation of $W_{E}$ for every $i$. Thus $t=t'$ and, renumbering the segments if necessary, we can assume that $(\rho(\Delta'_{i})|_{W_{E}}, N(\Delta_{i}'))=(\rho(\Delta_{i}), N(\Delta_{i}))$. Hence the segments $\Delta$ and $\Delta'$ are of same length, and if $\Delta_{i}=[\nu_{E}^{a_{i}}\sigma,\nu_{E}^{b_{i}}\sigma]$, then $\Delta_{i}'$ can be written as $[\nu_{F}^{a_{i}}\sigma_{i},\nu_{F}^{b_{i}}\sigma_{i}]$ where $\rec(\sigma_{i})|_{W_{E}}=\rec(\sigma)$. Therefore ${\rm bc}_{E/F}(\sigma_{i})=\sigma$, and as noted in eq.(\ref{eq_fiber:maps2}), this implies that $\sigma_{i}=\kappa^{k_{i}}\sigma'$ for some integer $k_{i}$. Thus we have that $\pi$ is a representation of the form described in eq.(\ref{eq_fiber:maps1}). 

The converse statement that every representation in $\aaa_{F}$ of the form described in eq.(\ref{eq_fiber:maps1}) lies in ${\rm bc}_{E/F}^{-1}(\Pi)$ follows directly from the definition of base change and the observation in eq.(\ref{eq_fiber:maps2}). 
\end{proof}

\subsubsection{}
\begin{remark}\label{rem_fiber:maps}
Note that in the statement of Lemma \ref {fiber:maps}(1), we have $\sigma'\ncong \kappa^{i}\sigma'$ for any $i\in \{1,\dots, d-1\}$. Indeed otherwise, by local Langlands correspondence and Clifford theory, this contradicts the fact that $\Pi$ is rigid. Similarly, in the statement of Lemma \ref {fiber:maps}(2), we have $\sigma'\ncong (\sigma')^{\gamma^{i}}$ for any $i\in \{1,\dots, d-1\}$. 

In particular, if $s$ is the size of the multi-set $\mathfrak m$, then the cardinality of the set ${\rm bc}_{E/F}^{-1}(\Pi)$ and ${\rm ai}_{E/F}^{-1}(\Pi)$ in the respective situations is $d^{s}$.
\end{remark}

\subsubsection{The case of generic rigid representations}
\begin{lemma}\label{lem_fiber:maps_gen}
\begin{enumerate}
\item Suppose that $\Pi=L(\mathfrak m)$ be rigid and in the image of the base change map.  Then $\Pi$ is generic if and only if every representation in ${\rm bc}_{E/F}^{-1}(\Pi)$ is generic. 
\item Suppose that $\Pi=L(\mathfrak m)$ be rigid and in the image of the automorphic induction map. Then $\Pi$ is generic if and only if every representation in ${\rm ai}_{E/F}^{-1}(\Pi)$ is generic. 
\end{enumerate}
\end{lemma}
\begin{proof}
Suppose first that $\Pi$ is generic. Let $\mathfrak m=\{\Delta_{1},\dots, \Delta_{s}\}\in \sigma^{\mathbb Z}$. Fix a $\sigma'\in {\rm bc}_{E/F}^{-1}(\sigma)$. By \cite[Theorem 9.7]{Z}, we have $\Delta_{i}\nprec\Delta_{j}$ for every $i$ and $j$. Thus $\kappa^{k_{i}}(\Delta_{i})_{(\sigma')}\nprec\kappa^{k_{j}}(\Delta_{j})_{(\sigma')}$ for any integers $k_{i}$ and $k_{j}$, and so by \cite[Theorem 9.7]{Z} and Lemma \ref{fiber:maps}, every representation in ${\rm bc}_{E/F}^{-1}(\Pi)$ is generic. The converse is obtained in a similar manner by applying \cite[Theorem 9.7]{Z} to a rigid representation in the fiber.
\end{proof}

\subsubsection{}The following result is an immediate corollary of Lemma \ref{fiber:maps}:
\begin{corr}\label{corr_fiber:maps}
\begin{enumerate}
\item Let $\Pi\in \mathcal L_{E}$. Then ${\rm bc}_{E/F}^{-1}(\Pi)\subset \mathcal L_{{\rm ind},F}$. 
\item Let $\Pi\in \mathcal L_{F}$. Then ${\rm ai}_{E/F}^{-1}(\Pi)\subset \mathcal L_{{\rm ind},E}$.  
\end{enumerate}
\qed
\end{corr}

\subsection{Estimates for Klyachko types}
\subsubsection{}
\begin{defn}\label{defn_fiber:maps}
\begin{enumerate}
\item Suppose $\Pi={\rm bc}_{E/F}(\pi)$ for some $\pi\in\aaa_{F}(n)$, is rigid, and admits a Klyachko model. Set $\mathcal H_{\pi,r}=\Hom_{H_{2k,r}}(\pi,\psi)$ (where $2k+r=n$). Define 
\begin{equation*}
d_{\Pi,{\rm bc}}=\sum _{\pi\in {\rm bc}_{E/F}^{-1}(\Pi)}\dim_{\mathbb C}(\mathcal H_{\pi,r(\Pi)}).
\end{equation*}
\item Suppose $\Pi={\rm ai}_{E/F}(\pi)$ for some $\pi\in\aaa_{E}(n)$, is rigid, and admits a Klyachko model. Set $\mathcal H_{\pi,r}=\Hom_{H_{2k,r}}(\pi,\psi)$ (where $2k+r=n$). Define 
\begin{equation*}
d_{\Pi,{\rm ai}}=\sum _{\pi\in {\rm ai}_{E/F}^{-1}(\Pi)}\dim_{\mathbb C}(\mathcal H_{\pi,\frac{r(\Pi)}{d}}).
\end{equation*}
\end{enumerate}
\end{defn}
Since an irreducible representation admits a Klyachko model with multiplicity at most one (\cite[Theorem 1]{OS2}), the integer $d_{\Pi,{\rm bc}}$ (or $d_{\Pi,{\rm ai}}$) is equal to the number of elements in the preimage of the respective maps which have the corresponding Klyachko type. For example if $\Pi=L(\mathfrak m)$ is a rigid generic representation in the image of the base change map, then by Lemma \ref{lem_fiber:maps_gen} and Remark \ref{rem_fiber:maps}, $d_{\Pi,{\rm bc}}=d^{s}$ where $s$ is the size of the multi-set $\mathfrak m$. 
\begin{remark}
Note that every cuspidal representation in the support of the representation $\Pi$ in Definition \ref{defn_fiber:maps}(2) lies in $\aaa_{F}(m)$ for some $m$ such that $d|m$. This fact along with Frobenius reciprocity implies that the Klyachko type of $\Pi$ has to be divisible by $d$. 
\end{remark}

\subsubsection{}
\begin{lemma}\label{lem_fiber:maps_symp}
\begin{enumerate}
\item Suppose that $\Pi=L(\mathfrak m)\in \mathcal L_{E}\cap \aaa_{E}(2n)$ such that $\Pi$ admits the symplectic model.  Further suppose that $\Pi$ is in the image of the base change map. Denote by $s$ the size of the multi-set $\mathfrak m$. Then
\begin{equation*}
d_{\Pi,{\rm bc}}=d^{\frac{s}{2}}.
\end{equation*}
\item Suppose that $\Pi=L(\mathfrak m)\in \mathcal L_{F}\cap \aaa_{F}(2dn)$ such that $\Pi$ admits the symplectic model. Further suppose that $\Pi$ is in the image of the automorphic induction map. Denote by $s$ the size of the multi-set $\mathfrak m$. Then
\begin{equation*}
d_{\Pi,{\rm ai}}=d^{\frac{s}{2}}.
\end{equation*}
\end{enumerate}
\end{lemma}
\begin{proof}
Since $r(\Pi)=0$, by Theorem \ref{thm: kly lad} $s$ is even. Let $s=2s'$ and let $\mathfrak m=\{\Delta_{1},\dots, \Delta_{2s'}\}$. Appealing to Theorem \ref{thm: kly lad} again, we get that $\Delta_{2j+1}=\nu\Delta_{2j+2}$ for every $j=0,\dots, s'-1$.  Let $\pi\in {\rm bc}_{E/F}^{-1}(\Pi)$. Then by Lemma \ref {fiber:maps}(1) $\pi$ is a representation of the form described in eq.(\ref{eq_fiber:maps1}), and by \cite[Lemma 5.9]{MOS}, the representations $\kappa^{i-1}L(\mathfrak m_{i})$ has a symplectic model for every $i$, if $\pi$ does so. Theorem \ref{thm: kly lad} applied to each $\kappa^{i-1}L(\mathfrak m_{i})$ gives us that if $(\Delta_{1})_{(\sigma')}\in\mathfrak m_{i}$, then $\nu_{F}^{-1}(\Delta_{1})_{(\sigma')}\in\mathfrak m_{i}$ as well.  Since $\mathfrak m$ is a ladder multi-set, the segment $\nu_{F}^{-1}(\Delta_{1})_{(\sigma')}$ can only be equal to $(\Delta_{2})_{(\sigma')}$. An easy induction gives us that for all $j=1,\dots,s'$ if $(\Delta_{2j-1})_{(\sigma')}\in\mathfrak m_{i}$ for some $i$, then $(\Delta_{2j})_{(\sigma')}\in\mathfrak m_{i}$ as well.  Thus any multi-set $\mathfrak m_{i}$ is of the form 
\begin{equation*}
\{\Delta_{2j_{1}+1},\Delta_{2j_{1}+2},\dots, \Delta_{2j_{a}+1},\Delta_{2j_{a}+2}\}_{(\sigma')}
\end{equation*}
for some distinct integers $\{j_{1},j_{2},\dots,j_{a}\}\subset \{0,\dots, s'-1\}$. By the observation in Remark \ref{rem_fiber:maps}, we get that $d_{\Pi,{\rm bc}}=d^{\frac{s}{2}}$.
\end{proof}

\subsubsection{} Thus for a ladder representation $\Pi$ having Whittaker or symplectic models, the integers $d_{\Pi,{\rm bc}}$ and $d_{\Pi,{\rm ai}}$ depend only on the degree of the field extension and the size of the underlying multi-set when $\Pi$ is expressed using the Langlands classification. For other Klyachko models this is not the case. For example consider the ladder representations $\Pi_{1}=L([\nu_{E}^{2},\nu_{E}^{3}],[\nu_{E},\nu_{E}^{2}],[1,\nu_{E}])$ and $\Pi_{2}=L([\nu_{E}^{3},\nu_{E}^{4}],[\nu_{E},\nu_{E}^{2}],[1,\nu_{E}])$. By Theorem \ref{thm: kly lad} both $\Pi_{1}$ and $\Pi_{2}$ admit Klyachko models and $r(\Pi_{1})=r(\Pi_{2})=2$. It is easy to check using Lemma \ref{fiber:maps}, Theorem \ref{thm: kly lad}, and \cite[Proposition 13.4]{MOS} that $d_{\Pi_{1},{\rm bc}}> d_{\Pi_{2},{\rm bc}}$. Similar examples can be constructed in the case of automorphic induction as well. 

\subsubsection{}However we can say the following: 
\begin{lemma}\label{lem_fiber:maps_mixed}
\begin{enumerate}
\item Suppose that $\Pi=L(\mathfrak m)\in \mathcal L_{E}$ such that $\Pi$ admits a Klyachko model other than the Whittaker and the symplectic models. Further suppose that $\Pi$ is in the image of the base change map. Denote by $s$ the size of the multi-set $\mathfrak m$. Then
\begin{equation*}
d^{\frac{s}{2}} \leq d_{\Pi,{\rm bc}}.
\end{equation*}
\item Suppose that $\Pi=L(\mathfrak m)\in \mathcal L_{F}$ such that $\Pi$ admits a Klyachko model other than the Whittaker and the symplectic models. Further suppose that $\Pi$ is in the image of the automorphic induction map. Denote by $s$ the size of the multi-set $\mathfrak m$. Then
\begin{equation*}
d^{\frac{s}{2}} \leq d_{\Pi,{\rm ai}}.
\end{equation*}
\end{enumerate}
\end{lemma}
\begin{proof}
Let us first consider the case when $\Pi\in  \mathcal L_{p,E}$. 
Suppose first that $s=2s'+1$ for some integer $s'$ and let $\mathfrak m=\{\Delta_{1},\dots,\Delta_{2s'+1}\}$. 
Let $\pi$ be a representation of the form described in eq.(\ref{eq_fiber:maps1}) such that each $\mathfrak m_{i}$ be a multi-set of the form 
\begin{equation*}
\{\Delta_{2i_{1}},\Delta_{2i_{1}+1},\dots, \Delta_{2i_{a}},\Delta_{2i_{a}+1}\}_{(\sigma')}
\end{equation*}
or 
\begin{equation*}
\{\Delta_{1},\Delta_{2i_{1}},\Delta_{2i_{1}+1},\dots, \Delta_{2i_{a}},\Delta_{2i_{a}+1}\}_{(\sigma')}
\end{equation*}
for some distinct integers $\{i_{1},i_{2},\dots,i_{a}\}\subset \{1,\dots, s'\}$, and $\mathfrak m_{1}+\cdots+\mathfrak m_{d}=\mathfrak m_{(\sigma')}$. Note that by Lemma \ref{fiber:maps}(1), the representation $\pi\in {\rm bc}_{E/F}^{-1}(\Pi)$ while by Theorem \ref{thm: kly lad} and \cite[Proposition 13.4]{MOS} it admits a Klyachko model with $r(\pi)=r(\Pi)$. By the observation in Remark \ref{rem_fiber:maps}, we get that $d_{\Pi,{\rm bc}}\geq d^{s'+1}$ which proves the lemma when $s$ is an odd integer. The case when $s$ is even is dealt with similarly which proves the statement for all proper ladders.

Now let $\Pi\in \mathcal L_{E}$. Using Proposition \ref{prop: prop} write $\Pi=\Pi_{1}\times\cdots\times \Pi_{k}$ where each $\Pi_{i}\in \mathcal L_{p,E}$ for $1\leq i\leq k$. It follows easily from Lemma \ref{fiber:maps}(1) that 
\begin{equation}\label{eqn_pr_lad}
{\rm bc}_{E/F}^{-1}(\Pi)=\{\pi_{1}\times \cdots \times \pi_{k}\mid \pi_{i}\in {\rm bc}_{E/F}^{-1}(\Pi_{i}), \ \forall \ 1\leq i \leq k\}.
\end{equation}
Using eq.(\ref{eqn_pr_lad}) and Theorem \ref{thm: kly lad}(2) we get that $d_{\Pi,{\rm bc}}\geq \prod_{i=1}^{k}d_{\Pi_{i},{\rm bc}}$. The statement for proper ladders proved above now implies the statement for $\Pi\in \mathcal L_{E}$.
\end{proof}

\subsubsection{} We have the following theorem summarizing the results of this section.

\begin{theorem}\label{prop_fiber:maps_all}
\begin{enumerate}
\item Suppose that $\Pi=L(\mathfrak m)\in \mathcal L_{E}$ such that $\Pi$ admits a Klyachko model and is in the image of the base change map. Denote by $s$ the size of the multi-set $\mathfrak m$. Then the set ${\rm bc}_{E/F}^{-1}(\Pi)$ has cardinality $d^{s}$ and 
\begin{equation*}
\frac{1}{2} \leq \frac{{\rm log}_{d}(d_{\Pi,{\rm bc}})}{s}\leq 1.
\end{equation*}
If $\Pi$ has the symplectic model, then the first inequality is an equality. The second inequality is an equality if and only if $\Pi$ has the Whittaker model.
\item Suppose that $\Pi=L(\mathfrak m)\in \mathcal L_{F}$ such that $\Pi$ admits a Klyachko model and is in the image of the automorphic induction map. Denote by $s$ the size of the multi-set $\mathfrak m$. Then the set ${\rm ai}_{E/F}^{-1}(\Pi)$ has cardinality $d^{s}$ and 
\begin{equation*}
\frac{1}{2} \leq \frac{{\rm log}_{d}(d_{\Pi,{\rm ai}})}{s}\leq 1.
\end{equation*}
If $\Pi$ has the symplectic model, then the first inequality is an equality. The second inequality is an equality if and only if $\Pi$ has the Whittaker model.
\end{enumerate}
\qed
\end{theorem}

\begin{remark}
It would be an interesting problem to find invariants for ladder representations, and more generally for rigid representations, which determine the integers $d_{\Pi,{\rm bc}}$ (resp., $d_{\Pi,{\rm ai}}$) completely for $\Pi$ with a given Klyachko type, and study their asymptotic behavior in the manner of Theorem \ref{prop_fiber:maps_all}.
\end{remark}

\subsubsection{Example of a Speh representation}
Let $\Pi=L(\mathfrak m)\in \aaa_{E}$ be a Speh representation in the image of the base change map and as above, let $|\mathfrak m|=s$. The simple structure of the representation allows us to obtain a precise value for the integer $d_{\Pi,{\rm bc}}$ in this case for any Klyachko model. If $s$ is even, then by Theorem \ref{thm: kly lad}(1) $\Pi$ has the symplectic model. Thus in this case $d_{\Pi,{\rm bc}}=d^{\frac{s}{2}}$ (by Lemma \ref{lem_fiber:maps_symp}(1)). So let $s=2s'+1$. Let $\Delta$ be an arbitrary segment in $\mathfrak m$ and let the integer $m$ be such that $L(\Delta)\in \aaa_{E}(m)$. Then applying Theorem \ref{thm: kly lad}(1) again we get that $\Pi$ admits a Klyachko model and $r(\Pi)=m$. Arguing as in the proof of Lemma \ref{lem_fiber:maps_symp}, in this case it is easy to see that 
\begin{equation*}
d_{\Pi,{\rm bc}}=(s'+1)d^{(s'+1)}-s'd^{s'}. 
\end{equation*}
For a Speh representation $\Pi\in \aaa_{F}$ in the image of the automorphic induction map, the integer $d_{\Pi,{\rm ai}}$ can be calculated similarly. 

\proof[Acknowledgements] 
The authors wish to thank Omer Offen for several helpful comments. The first author would also like to thank the Tata Institute of Fundamental Research, Mumbai, for providing an excellent work environment during which this work was completed.

\end{document}